\documentclass[a4paper,11pt]{article}
\usepackage{amsfonts,amsthm,amsmath,amssymb,stmaryrd}
\usepackage{enumitem}
\usepackage{mathtools}
\usepackage{tikz}
\usepackage{tikz-cd}
\usepackage{graphicx}
\usepackage{color}

\definecolor{mycitecolor}{rgb}{0,0,0.8}
\definecolor{myrefcolor}{rgb}{0,0,0.8}

\usepackage[bookmarks=true]{hyperref}
\hypersetup{colorlinks, linkcolor=myrefcolor, citecolor=mycitecolor} 

\usepackage{url}

\newcommand{\btk}{\begin{tikzcd}}
\newcommand{\etk}{\end{tikzcd}}
\newcommand{\bc}{\begin{center}}
\newcommand{\ec}{\end{center}}

\newcommand{\Set}{\operatorname{\mathsf{Set}}}
\newcommand{\Cat}{\operatorname{\mathsf{Cat}}}
\newcommand{\FinSet}{\operatorname{\mathsf{FinSet}}}

\newcommand{\FactSys}{\operatorname{\mathsf{FactSys}}}
\newcommand{\Data}{\operatorname{\mathsf{DecData}}}

\newcommand{\Rel}{\operatorname{\mathsf{Rel}}}
\newcommand{\LinRel}{\operatorname{\mathsf{LinRel}}}
\newcommand{\CospanAlg}{\operatorname{\mathsf{CospanAlg}_\mathsf{LFS}}}
\newcommand{\calg}{\operatorname{\mathsf{CospanAlg}}}
\newcommand{\HypOF}{\operatorname{\mathsf{Hyp}_{OF}}}
\newcommand{\Hyp}{\operatorname{\mathsf{Hyp}}}
\newcommand{\SemiAlgRel}{\operatorname{\mathsf{SARel}}}
\newcommand{\sarel}{\operatorname{\mathsf{SARel}}_\mathrm{H}}
\newcommand{\Dynam}{\operatorname{\mathsf{Dynam}}}
\newcommand{\comp}{\mathsf{comp}}
\newcommand{\algtohyp}{\Phi}
\newcommand{\hyptoalg}{\mathsf{Alg}}
\newcommand{\mon}{\operatorname{\mathsf{Mon}}}
\newcommand{\frob}{\operatorname{\mathsf{Frob}}}
\newcommand{\Frob}{\operatorname{\mathsf{Frob}}}
\newcommand{\graph}{\operatorname{\mathsf{Gr}}}
\newcommand{\op}{\mathrm{op}}
\newcommand{\cp}{\mathord{\fatsemi}}
\newcommand{\comma}{_}

\DeclareMathOperator{\colim}{colim}

\DeclareMathOperator{\id}{id}

\newcommand{\Corel}{\mathsf{Corel}}
\newcommand{\Cospan}{\mathsf{Cospan}}
\newcommand{\FCospan}{F\Cospan}
\newcommand{\DCospan}{D\Cospan}
\newcommand{\FCorel}{F\Corel}

\DeclareMathOperator{\Lax}{\mathsf{Lax}}
\newcommand{\Ob}{\mathsf{Ob}}
\newcommand{\Lan}{\mathsf{Lan}}
\newcommand{\Kan}{\mathsf{Kan}}

\newcommand{\C}{\mathcal{C}}
\newcommand{\E}{\mathcal{E}}
\newcommand{\M}{\mathcal{M}}

\newcommand{\I}{\mathcal{I}}

\newcommand{\K}{\mathcal{K}}
\renewcommand{\H}{\mathcal{H}}

\newcommand{\R}{\mathbb{R}}

\newcommand{\define}[1]{\emph{#1}}

\theoremstyle{theorem}
\newtheorem{thm}{Theorem}[section]
\newtheorem*{thm*}{Theorem}
\newtheorem{prop}[thm]{Proposition}
\newtheorem{lemma}[thm]{Lemma}
\newtheorem{cor}[thm]{Corollary}

\theoremstyle{definition}
\newtheorem{defn}[thm]{Definition}

\theoremstyle{remark}
\newtheorem{ex}[thm]{Example}
\newtheorem{rmk}[thm]{Remark}

  \usetikzlibrary{
  	cd,
  	math,
  	decorations.markings,
		decorations.pathreplacing,
  	positioning,
  	arrows.meta,
  	shapes,
		shadows,
		shadings,
  	calc,
  	fit,
  	quotes,
  	intersections,
    circuits,
    circuits.ee.IEC
  }

\tikzset{
	spider diagram/.style={
		every to/.style={out=0, in=180, draw, thick},
		thick
	},
	dot size/.store in=\dotsize,
	dot size = 5pt,
	dot fill/.store in=\dotfill,
	dot fill = black,
	leg length/.store in=\leglen,
	leg length = 15pt,
	baby/.style={dot size = 2pt, leg length = 6pt},
	young/.style={dot size = 3pt, leg length = 10pt},
	special spider/.code n args={4}{
		\pgfkeysalso{circle, draw, thick, inner sep=0, fill=\dotfill, minimum width=\dotsize,
  		prefix after command={\pgfextra{\let\fixname\tikzlastnode}},
  		append after command={\pgfextra{
  			\ifnum #1=0{} \else {\foreach \i in {1,...,#1} {
					\tikzmath{\anglei={-90*(#1+1-2*\i)/#1};}
  				\draw [thick]
						(\fixname) .. controls
						($(\fixname.center)-(\anglei:#3/3)$) and ($(\fixname.center)-(\anglei:#3*2/3)$) ..
						({$(\fixname.center)-(\anglei:#3*2/3)$}-|{$(\fixname.center)-(#3,0)$}) coordinate (\fixname_in\i);
  			}}\fi
  			\ifnum #2=0{} \else {\foreach \i in {1,...,#2} {
					\tikzmath{\anglei={90*(#2+1-2*\i)/#2};}
  				\draw [thick]
						(\fixname.center) .. controls
						($(\fixname.center)+(\anglei:#4/3)$) and ($(\fixname.center)+(\anglei:#4*2/3)$) ..
						({$(\fixname.center)+(\anglei:#4*2/3)$}-|{$(\fixname.center)+(#4,0)$}) coordinate (\fixname_out\i);
  			}}\fi
  		}}
		}
	},
	spider/.code 2 args={
		\pgfkeysalso{special spider={#1}{#2}{\leglen}{\leglen}}
	}
}

\usetikzlibrary{arrows,backgrounds,circuits,circuits.ee.IEC,shapes,fit,matrix,decorations.markings, positioning}
\tikzstyle{none}=[inner sep = 0]

\definecolor{mygreen}{rgb}{0.7,1,0.7}
\definecolor{myyellow}{rgb}{1,1,0.6}
\tikzstyle{species}=[circle,fill=myyellow,draw=black,scale=2.15, inner sep = 2pt]
\tikzstyle{reaction}=[rectangle,fill=mygreen,draw=black,scale=2]
\tikzstyle{inarrow}=[->, >=stealth, shorten >=.03cm,line width=1.5]
\tikzstyle{empty}=[circle,fill=none, draw=none]
\tikzstyle{inputdot}=[circle,fill=purple,draw=purple, scale=.25]
\tikzstyle{inputarrow}=[->,draw=purple, shorten >=.05cm]
\tikzstyle{simple'}=[-,draw=purple,line width=1.000]
\tikzstyle{arrow}=[-,draw=black,postaction={decorate},decoration={markings,mark=at position .5 with {\arrow{>}}},line width=1.200]

\pgfdeclarelayer{edgelayer}
\pgfdeclarelayer{nodelayer}
\pgfsetlayers{edgelayer,nodelayer,main}

\tikzstyle{main node} =[circle,fill=white!20,draw,font=\sffamily\Large\bfseries]
\tikzstyle{terminal}=[circle,fill=white!20,draw,font=\sffamily\Large\bfseries,color=purple,fill=none]
\tikzstyle{none}=[inner sep=0pt]
\tikzstyle{dot}=[circle,fill=black,draw=black]
\tikzstyle{simple}=[-,draw=black,line width=1.500]

\usetikzlibrary{arrows}
\usetikzlibrary{decorations.markings}

\begin{document}

\title{A recipe for black box functors}

\author{{\scriptsize BRENDAN FONG AND MARU SARAZOLA}}

\date{}

\maketitle
\begin{abstract}
  The task of constructing compositional semantics for network-style diagrammatic languages, such as electrical circuits or chemical reaction networks, has been dubbed the black boxing problem, as it gives semantics that describes the properties of each network that can be observed externally, by composition, while discarding the internal structure. One way to solve these problems is to formalise the diagrams and their semantics using hypergraph categories, with semantic interpretation a hypergraph functor, called the black box functor, between them.
  Reviewing a principled method for constructing hypergraph categories and functors, known as decorated corelations, in this paper we construct a category of \emph{decorating data}, and show that the decorated corelations method is itself functorial, with a universal property characterised by a left Kan extension. We then argue that the category of decorating data is a good setting in which to construct any hypergraph functor, giving a new construction of Baez and Pollard's black box functor for reaction networks as an example.
\end{abstract}


\section{Introduction}

From chemical reaction networks to tensor networks to finite state automata, network diagrams are often used to represent and reason about interconnected systems. What makes such a language convenient, however, is not just that these diagrams are intuitive to read and work with: it's that the notion of networking itself has meaning in the relevant semantics of the diagrams---the chemical or computational systems themselves.

More formally, recent work has used monoidal categories, and in particular \emph{hypergraph} categories, to describe the algebraic structure of such systems, including electrical circuits, signal flow graphs, Markov processes, and automata, among many others \cite{BF,BSZ17,ASW11,BFP16,GKS17}. In this approach, diagrams are formalised as morphisms in a hypergraph category which represents the syntax of the language, and they are interpreted in another hypergraph category which models the semantics of the language. What matters then, is that this process of interpretation preserves the network operations: that this map forms what is known as a hypergraph functor.

In these hypergraph categories, the objects model interface or boundary types, and semantic interpretation often has the effect of hiding internal structure, and reducing the combinatorial, network-style diagram description of a system to the data that can be obtained via interaction, or composition, with other systems. In other words, semantic interpretation has the effect of wrapping the network, say an electrical circuit, in a `black box'. We hence, informally, call a hypergraph functor that describes the semantics of a system a black box functor. This paper describes a general method for constructing such functors.

Let's consider an example. In their paper ``A compositional framework for reaction networks'' \cite{BP17}, Baez and Pollard define a black box functor for reaction networks. Reaction networks, also known as stochastic Petri nets, were developed to describe systems of chemical reactions and their dynamics. Here is an example of a reaction network:
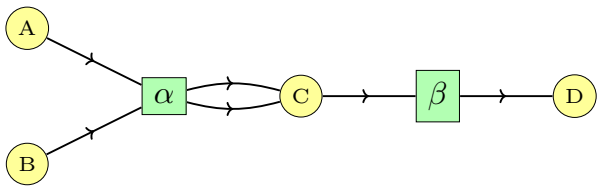
\[
\scalebox{0.6}{
\begin{tikzpicture}
	\begin{pgfonlayer}{nodelayer}
		\node [style=reaction] (2) at (-3, 0) {\small{$\alpha$}};
		\node [style=species] (3) at (0, 0) {\tiny{C}};
		\node [style=reaction] (4) at (3, 0) {\small{$\beta$}};
		\node [style=species] (0) at (-6, 1.5) {\tiny{A}};
		\node [style=species] (1) at (-6, -1.5) {\tiny{B}};
		\node [style=species] (5) at (6, 0) {\tiny{D}};
	\end{pgfonlayer}
	\begin{pgfonlayer}{edgelayer}
		\draw [style=arrow] (3) to (4);
		\draw [style=arrow] (4) to (5);
		\draw [style=arrow, bend left=15, looseness=1.00] (2) to (3);
		\draw [style=arrow, bend right=15, looseness=1.00] (2) to (3);
		\draw [style=arrow] (0) to (2);
		\draw [style=arrow] (1) to (2);
	\end{pgfonlayer}
\end{tikzpicture}
}
\]
Here A, B, C, and D represent chemical species, such as carbon dioxide or water, and $\alpha$ and $\beta$ represent chemical reactions; for example $\alpha$ represents the reaction $A+B \longrightarrow 2C$.

To consider the depicted network as an \emph{open} reaction network, and as a morphism in a category, we annotate this data with left and right boundaries, which one might consider as inputs and outputs. For example, in the following
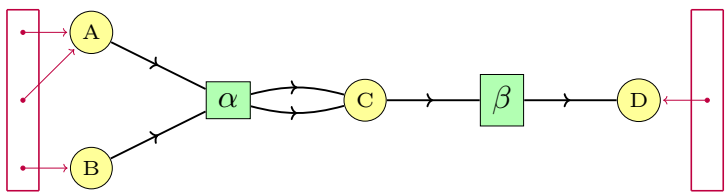
\[
\scalebox{0.6}{
\begin{tikzpicture}
	\begin{pgfonlayer}{nodelayer}
		\node [style=reaction] (2) at (-3, 0) {\small{$\alpha$}};
		\node [style=species] (3) at (0, 0) {\tiny{C}};
		\node [style=reaction] (4) at (3, 0) {\small{$\beta$}};
		\node [style=species] (0) at (-6, 1.5) {\tiny{A}};
		\node [style=species] (1) at (-6, -1.5) {\tiny{B}};
		\node [style=species] (5) at (6, 0) {\tiny{D}};
        \node [style=inputdot] (a1) at (-7.5,0) {};
        \node [style=inputdot] (a2) at (-7.5,1.5) {};
        \node [style=inputdot] (b) at (-7.5,-1.5) {};
        \node [style=inputdot] (d) at (7.5,0) {};
	\end{pgfonlayer}
	\begin{pgfonlayer}{edgelayer}
		\draw [style=arrow] (3) to (4);
		\draw [style=arrow] (4) to (5);
		\draw [style=arrow, bend left=15, looseness=1.00] (2) to (3);
		\draw [style=arrow, bend right=15, looseness=1.00] (2) to (3);
		\draw [style=arrow] (0) to (2);
		\draw [style=arrow] (1) to (2);
        \draw [style=inputarrow] (a1) to (0);
        \draw [style=inputarrow] (a2) to (0);
        \draw [style=inputarrow] (b) to (1);
        \draw [style=inputarrow] (d) to (5);
        \draw [style=simple'] (7.15,2) to (7.15,-2);
        \draw [style=simple'] (7.85,2) to (7.85,-2);
        \draw [style=simple'] (7.15,2) to (7.85,2);
        \draw [style=simple'] (7.15,-2) to (7.85,-2);
        \draw [style=simple'] (-7.15,2) to (-7.15,-2);
        \draw [style=simple'] (-7.85,2) to (-7.85,-2);
        \draw [style=simple'] (-7.15,2) to (-7.85,2);
        \draw [style=simple'] (-7.15,-2) to (-7.85,-2);
	\end{pgfonlayer}
\end{tikzpicture}
}
\]

If the right boundary of one open network coincides with the left boundary of another, we may compose them by identifying all chemical species that share the same boundary annotation. Although we could represent these reaction networks directly as some form of labelled graph, and indeed Baez and Pollard do so, the core of the problem of describing black box semantics is better illustrated if we jump straight to representing the system as a vector field on the space of concentrations of the chemical species in the network. As we shall see, this idea turns reaction networks into the morphisms of a hypergraph category known as $\Dynam$. In particular, an object in $\Dynam$ is a finite set $X$, and a morphism $X \to Y$ in $\Dynam$ is a cospan of finite sets $X \to N \leftarrow Y$ together with a suitably well-behaved vector field on $N$. Here we think of each finite set as a set of chemical species, the cospan as marking the species in $N$ that can be externally controlled as inputs and outputs, and the vector field as describing the dynamics of the chemical system.

Reaction networks give rise to a system of coupled differential equations, whose solutions in turn represent the ways in which the concentrations of the chemical species in the system can vary over time. Baez and Pollard describe, in particular, the \emph{steady state} semantics: the concentrations of species that are stable with respect to supplying an inflow or outflow of chemicals at the boundary at a certain rate. Each steady state is described by a tuple $(c_i,f_i,c_o,f_o)$, whose entries represent, respectively, the concentrations at the left boundary, or input, the inflow, the concentrations at the right boundary, and the outflow. Thus the set of all steady states forms a binary relation between an input and an output space of concentrations and flows. Writing $\Rel$ for the category of sets of binary relations, Baez and Pollard show that this defines a strong symmetric monoidal functor $\blacksquare\colon \Dynam \to \Rel$.

In this paper we will provide a general recipe for constructing such functors. In doing so, we will provide a streamlined proof of the functoriality of this functor $\blacksquare$, and in fact show, moreover, that we can understand it as a hypergraph functor.

Baez and Pollard use what is known as a decorated cospans construction to define the domain of their functor, $\Dynam$. The main theme of this paper is a careful study of a generalisation of this construction, introduced in \cite{Fong18}, known as \emph{decorated corelations}. Decorated cospans constructs a hypergraph category from a finitely cocomplete category $\C$ and a lax symmetric monoidal functor $F\colon (\C,+) \to (\Set,\times)$. In the above example $\Dynam$, the finitely cocomplete category is the category $\FinSet$ of finite sets, and $F$ is the functor that maps a finite set $N$ to the set of suitably well-behaved vector fields on $\R^N$. Decorated corelations generalises this by also requiring a factorisation system $(\E,\M)$ on $\C$, and extending $F$ to a functor on a certain subcategory $\C\cp\M^\op$ of $\Cospan(\C)$. As we shall see, this method suffices to construct, up to equivalence, all hypergraph categories and hypergraph functors. Moreover, it can be used to represent the data of a hypergraph category in an efficient way that makes the data easy to work with.

Our first main result is the functoriality of this construction. Indeed, we define a category $\Data$ whose objects are \emph{decorating data}: the tuples $(\C,(\E,\M),F)$ required for the decorated corelations construction. Write $\Hyp$ for the category whose objects are hypergraph categories and morphisms are hypergraph functors.
\begin{thm*}
  The decorated corelations construction defines a functor
  \[
    (-)\Corel \colon \Data \to \Hyp.
  \]
\end{thm*}

We will prove this theorem by a characterisation, interesting in its own right, of the decorated corelations functor $(-)\Corel$ in terms of left Kan extension. To do this, we make use of a full subcategory of $\Data$ that we call $\calg$, whose objects are finitely cocomplete categories $\C$ together with a lax symmetric monoidal functor $\Cospan(\C) \to \Set$.

\begin{thm*}
  The functor $(-)\Corel$ factors as the composite $\algtohyp \circ \Kan$, where these functors are part of adjunctions
\[
\btk[column sep=small]
\Hyp \ar[rr, shift left =5pt, "\hyptoalg"] &
  \scriptstyle\bot &
  \calg \ar[ll, shift left =5pt, "\algtohyp"] \ar[rr, shift left =5pt, "\iota"] &
  \scriptstyle\top &
  \Data.  \ar[ll,shift left =5pt, "\Kan"]
  \etk
\]
\end{thm*}
As a corollary of these investigations, we shall see that decorated corelations allow us to build, up to equivalence, any hypergraph category and any hypergraph functor.

Moreover, we shall see that this understanding allows a new proof of the functoriality of Baez and Pollard's black box functor. The key idea is that Baez and Pollard start with an object $D$ in $\Data$, and use $(-)\Corel$ to turn it into a hypergraph category $\Dynam$ in $\Hyp$. Then, after defining the hypergraph category of semialgebraic relations $\SemiAlgRel$, they directly construct the hypergraph functor $\blacksquare$. On the other hand, we show that using the functor $\hyptoalg$ and inclusion of categories $\iota$, we can reduce $\SemiAlgRel$ to an object $S$ in $\Data$ and construct a morphism in $\Data$ from $D$ to $S$. In short, we work on the right of the above diagram, instead of on the left. The properties of $\Data$ ensure that this category is easier to work with, and the functor $(-)\Corel$ lifts this morphism in $\Data$ to the desired black box functor. We hope that the simplicity of our proof provides a recipe for further work on constructing black box functors.

\paragraph{Outline}
The layout of this paper is as follows. Section \ref{section.background} contains the necessary background, recalling the definitions of the category of cospans $\Cospan(\C)$, of Frobenius monoids, and of hypergraph categories, along with some key examples that will be of use later in the paper.

Section \ref{section.decorating.data} deals with decorating data. After presenting the types of decorating data available in the literature, decorated cospans and corelations, we introduce a new category $\Data$ of decorating data which consists of all the relevant information one needs for decorating purposes. At the end of the section we state our first main theorem (Theorem~\ref{thm.deccorelfunctorial}): decorating corelations extends to a functor $(-)\Corel\colon \Data \to \Hyp$. The proof is deferred to Section \ref{section.corel.functor}.

In Section \ref{section.cospanalg}, we consider the full subcategory $\calg\subseteq\Data$. We show that this embedding admits a left adjoint, which we call $\Kan$ since it entails of taking left Kan extensions of the functors involved. The latter part of this section is mainly technical, and provides an explicit description of the functor $\Kan$ that proves useful for comparing it to previous constructions.

Section \ref{section.hyp.and.cospanalg} considers the restriction of $(-)\Corel$ to the category $\calg$, which we show to be  functorial. Section \ref{section.corel.functor} contains our main result, Thm.\ \ref{thm.kan}, which states that the functor $(-)\Corel\colon\Data\to\Hyp$ factors as $$\Data\xrightarrow{\Kan} \calg\xrightarrow{\algtohyp}\Hyp.$$ In particular, this implies Thm.\ \ref{thm.deccorelfunctorial}.

Section \ref{recover_hyp} shows that the functor $\algtohyp$ admits a left adjoint $\hyptoalg\colon\Hyp\to\calg$. Furthermore, the unit of this adjunction is a component-wise equivalence of hypergraph categories, which proves that every hypergraph category can be built from $\Data$ via the decorated corelations construction. Thus, our category $\Data$ contains all the necessary information for dealing with hypergraph categories. In the case that our hypergraph categories are objectwise free, we show our construction in fact recovers the hypergraph categories up to isomorphism. 

Finally, Section \ref{section.application} gives an application of our results. We briefly recall the main aspects of open dynamical systems, and of the black-boxing functor $\blacksquare\colon\Dynam\to\SemiAlgRel$ constructed in \cite{BP17}. We then show how this functor can be obtained in a simple and natural way through our results.


\paragraph{Acknowledgements}
We thank David Spivak and John Baez for useful conversations. BF was supported by USA AFOSR grants FA9550-14-1-0031 and FA9550-17-1-0058. MS was supported by Cornell University's Torng Prize Fellowship.

\section{Background}\label{section.background}

We quickly review the notions of cospan category, Frobenius monoid, and hypergraph category. A fuller introduction can be found in \cite{Fong16,FS}. Throughout this paper, we will assume all categories are essentially small.

\subsection{Cospan categories}

As observed by Benabou \cite{Ben67}, for any finitely cocomplete category $\C$, we can define a symmetric monoidal category $\Cospan(\C)$ whose objects are the same objects as in $\C$, and where a map $X\to Y$ is given by an isomorphism class of cospans $X\xrightarrow{i} N \xleftarrow{o} Y$. Recall that two cospans $X\xrightarrow{i} N \xleftarrow{o} Y$ and $X\xrightarrow{i'} N' \xleftarrow{o'} Y$ sharing the same feet are isomorphic if there exists an isomorphism between the apexes $N\xrightarrow{f} N'$ such that $f \circ i = i'$ and $f \circ o = o'$.

Cospans $X\xrightarrow{i} N \xleftarrow{o} Y$ and $Y\xrightarrow{i'} M \xleftarrow{o'} Z$ compose by using the pushout along the common foot $X\xrightarrow{j_Ni} P \xleftarrow{j_Mo'} Z$:
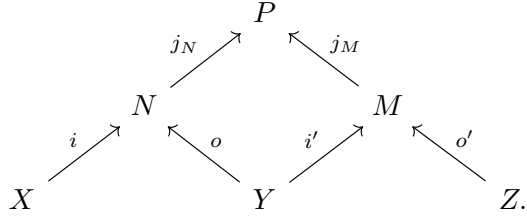
\[\btk
& & P
& & \\
& N\ar[ru,"j_N"] & & M\ar[lu,"j_M"'] & \\
X\ar[ru,"i"] & & Y\ar[lu,"o"']\ar[ru,"i'"] & & Z.\ar[lu,"o'"']
\etk\]

Since $\C$ has finite colimits, it can be given a symmetric monoidal structure with the coproduct and initial object playing the role of the tensor product and unit. Thus $\Cospan(\C)$ inherits this monoidal structure via the embedding $\C\hookrightarrow\Cospan(\C)$, an identity-on-objects functor taking a map $f\colon X\to Y$ in $\C$ to the cospan $\btk[cramped] X \ar[r,"f"] & Y & Y\ar[l,equal]\etk$, where we use the long equals sign to denote the identity map. We will sometimes abuse notation and refer to this cospan simply as $f$, and to its ``opposite'' cospan $\btk[cramped] Y\ar[r,equal] & Y & X\ar[l,"f"'] \etk$ as $f^\op$.

\subsection{Frobenius monoids}

\begin{defn}
A \define{special commutative Frobenius monoid} in a symmetric monoidal category $\C$ is an object $X$ together with maps
\[
\begin{tikzpicture}[spider diagram, dot fill=black]
	\node[spider={2}{1}, label={[below=.5]:{$\mu\colon X\otimes X\to X$}}] (a) {};
	\node[spider={0}{1}, label={[below=.5]:{$\eta\colon I\to X$}}, right=3 of a] (b) {};
	\node[spider={1}{2}, label={[below=.5]:{$\delta\colon X\to X\otimes X$}}, right=3 of b] (c) {};
	\node[spider={1}{0}, label={[below=.5]:{$\epsilon\colon X\to I$}}, right=3 of c] (d) {};
\end{tikzpicture}
\]
such that $(X,\mu,\eta)$ is a commutative monoid, i.e.
\[
  \begin{tikzpicture}[spider diagram,xscale=.8,yscale=.8]
	\node[spider={2}{1}] (a) {};
	\node[special spider={2}{1}{\leglen}{0}, left=.1 of a_in1] (b) {};
	\draw (a_in1) to (b.east);
	\draw (a_in2) to (b_in1|-a_in2);
	\node[spider={2}{1}, right=2.5 of b] (aa) {};
	\node[special spider={2}{1}{\leglen}{0}, left=.1 of aa_in2] (bb) {};
	\draw (aa_in2) to (bb.east);
	\draw (aa_in1) to (bb_in1|-aa_in1);
	\node at ($(b.west)!.5!(aa.east)$) {$=$};
	\node[spider={2}{1}, right=2 of aa_out1] (a) {};
	\node[special spider={0}{1}{\leglen}{0}, left=.1 of a_in1] (b) {};
	\draw (a_in1) to (b.east);
	\draw (a_in2) to ($(a_in2)-(\leglen,0)$);
	\coordinate[right=.5 of a_out1] (aa) {};
	\coordinate (bb) at ($(aa)+(3*\leglen,0)$) {};
	\draw (aa) to (bb);
	\node at ($(a_out1)!.5!(aa)$) {$=$};
	\node[spider={2}{1}, right=2 of bb] (a) {};
	\coordinate (b1) at ($(a_in1)-(.2,0)$);
	\coordinate (b2) at ($(a_in2)-(.2,0)$);
	\draw (a_in1) to (b1);
	\draw (a_in2) to (b2);
	\node[spider={2}{1}, right=1.5 of a_out1] (aa) {};
	\coordinate (bb1) at ($(aa_in1)-(.6,0)$);
	\coordinate (bb2) at ($(aa_in2)-(.6,0)$);
	\coordinate (bb) at ($(bb1)!.5!(bb2)$);
	\draw (bb1) to (aa_in2);
	\draw (bb2) to (aa_in1);
	\node at ($(a_out1)!.5!(bb)$) {$=$};
\end{tikzpicture}
\]
$(X,\delta,\epsilon)$ is a cocommutative comonoid, i.e.
\[
\begin{tikzpicture}[spider diagram,scale=.8]
	\node[spider={1}{2}] (a) {};
	\node[special spider={1}{2}{0}{\leglen}, right=.1 of a_out1] (b) {};
	\draw (a_out1) to (b_in1);
	\draw (a_out2) to (b_out1|-a_out2);
	\node[spider={1}{2}, right=1.5 of b] (aa) {};
	\node[special spider={1}{2}{0}{\leglen}, right=.1 of aa_out2] (bb) {};
	\draw (aa_out2) to (bb_in1);
	\draw (aa_out1) to (bb_out1|-aa_out1);
	\node at ($(b.west)!.5!(aa.east)$) {$=$};
	\node[spider={1}{2}, right=2.5 of aa] (a) {};
	\node[special spider={1}{0}{0}{\leglen}, right=.1 of a_out1] (b) {};
	\draw (a_out1) to (b_in1);
	\draw (a_out2) to ($(a_out2)+(\leglen,0)$);
	\coordinate[right=1.5 of a] (aa) {};
	\coordinate (bb) at ($(aa)+(3*\leglen,0)$) {};
	\draw (aa) to (bb);
	\node at ($(a_in1)!.5!(bb)$) {$=$};
	\node[spider={1}{2}, right=1.5 of bb] (a) {};
	\coordinate (b1) at ($(a_out1)+(.2,0)$);
	\coordinate (b2) at ($(a_out2)+(.2,0)$);
	\draw (a_out1) to (b1);
	\draw (a_out2) to (b2);
	\node[spider={1}{2}, right=1.5 of a] (aa) {};
	\coordinate (bb1) at ($(aa_out1)+(.6,0)$);
	\coordinate (bb2) at ($(aa_out2)+(.6,0)$);
	\coordinate (bb) at ($(bb1)!.5!(bb2)$);
	\draw (aa_out2) to (bb1);
	\draw (aa_out1) to (bb2);
	\node at ($(a_in1)!.45!(bb)$) {$=$};
\end{tikzpicture}
\]
and the Frobenius and special axioms are satisfied:
\[
\begin{tikzpicture}[spider diagram,scale=.8]
	\node[spider={1}{2}] (a) {};
	\node[spider={2}{1}, below right=\leglen*2/3 and 1 of a] (b) {};
	\coordinate (c1) at (a_out1-|b_out1);
	\coordinate (c2) at (a_in1|-b_in2);
	\draw (a_out2) -- (b_in1);
	\draw (a_out1) -- (c1);
	\draw (c2) -- (b_in2);
	\node[special spider={2}{1}{\leglen}{0}, right=2 of a_out2] (aa) {};
	\node[special spider={1}{2}{0}{\leglen}, right=.5 of aa] (bb) {};
	\draw (aa_out1) -- (bb_in1);
	\node[spider={1}{2}, right=1 of bb_out2] (aaa) {};
	\node[spider={2}{1}, above right=\leglen*2/3 and 1 of aaa] (bbb) {};
	\coordinate (ccc1) at (aaa_out2-|bbb_out1);
	\coordinate (ccc2) at (aaa_in1|-bbb_in1);
	\draw (aaa_out1) -- (bbb_in2);
	\draw (aaa_out2) -- (ccc1);
	\draw (ccc2) -- (bbb_in1);
	\coordinate (h1) at ($(b_out1)!.5!(aa_in1)$);
	\coordinate (h2) at ($(aaa_in1)!.5!(bb_out1)$);	
	\coordinate (h3) at ($(aa_out1)!.5!(bb_in1)$);
	\node at (h1|-aa) {$=$};
	\node at (h2|-aa) {$=$};
	\node[spider={1}{2}, right=3 of aaa_out1] (aaaa) {};
	\node[spider={2}{1}, right=1 of aaaa] (bbbb) {};
	\draw (aaaa_out1) -- (bbbb_in1);
	\draw (aaaa_out2) -- (bbbb_in2);
	\coordinate[right=2.5 of aaaa] (aaaaa) {};
	\coordinate[right=1 of aaaaa] (bbbbb) {};
	\draw (aaaaa) -- (bbbbb);
	\node at ($(bbbb_out1)!.5!(aaaaa)$) (h4) {$=$};
\end{tikzpicture}
\]
\end{defn}

\begin{ex}\label{cospan_is_hypergraph}
Every object $X$ in $\Cospan(\C)$ can be given a canonical special commutative Frobenius structure. The monoid structure is inherited from its canonical monoid structure in $(\C,+,\emptyset)$---that is, $\eta\coloneqq \ !\colon\emptyset\to X$ and $\mu\coloneqq[1_X,1_X]\colon X+X\to X$---through the embedding $\C\hookrightarrow\Cospan(\C)$. The comonoid structure is given by the opposite cospans; explicitly, the counit and coproduct maps are respectively
\[\btk
X\ar[r,equal] & X & \emptyset\ar[l,"\eta"']
\etk
\hspace{1cm}\text{and}\hspace{1cm}
\btk
X\ar[r,equal] & X & X+X. \ar[l,"\mu"']
\etk\]
\end{ex}

\subsection{Hypergraph categories}

We can consider monoidal categories in which every object has a Frobenius structure; these are due to Carboni and Walters \cite{Car91}. Write $\sigma$ for the braiding in a symmetric monoidal category.

\begin{defn}
  A \define{hypergraph category} is a symmetric monoidal category $(\H,\otimes,I)$ whose every object $X$ is equipped with a Frobenius structure $(\mu_X,\eta_X,\delta_X,\epsilon_X)$ in a manner that is compatible with the monoidal structure in $\H$; that is, such that the Frobenius structure on $X \otimes Y$ is
  \[
  \Big( (\mu_X \otimes \mu_Y) \circ (1 \otimes \sigma \otimes 1),\enspace
  \eta_X \otimes \eta_Y, \enspace
  (1 \otimes \sigma \otimes 1) \circ (\delta_X \otimes \delta_Y),\enspace
 \epsilon_X \otimes \epsilon_Y
  \Big),
\]
and such that the Frobenius structure on $I$ is $(\rho_I^{-1},\id_I,\rho_I,\id_I)$, where $\rho$ denotes the unitor in $(\H,\otimes,I)$.

A hypergraph functor is a strong symmetric monoidal functor between hypergraph categories that preserves the Frobenius structures present---that is, such that if the Frobenius structure on $X$ is $(\mu_X,\eta_X,\delta_X,\epsilon_X),$ then that on $FX$ must be
\[
  \left(F\mu_X \circ \varphi_{X,X},\enspace
  F\eta_X \circ \varphi_I, \enspace
  \varphi^{-1}_{X,X} \circ F\delta_X,\enspace
  \varphi_I^{-1} \circ F\epsilon_X\right).
\]

\end{defn}

\begin{ex}
As we saw in Example \ref{cospan_is_hypergraph}, every object in $\Cospan(\C)$ can be given a Frobenius structure, and one can further prove that these are compatible in the required sense, making $\Cospan(\C)$ into our prototypical example of hypergraph category.
\end{ex}

\begin{ex}\label{linrel}
Consider the symmetric monoidal category $\LinRel$, whose objects are finite dimensional $\mathbb{R}$-vector spaces with a chosen basis (or in other words, $\mathbb{R}^n$,) and where maps $\mathbb{R}^m\to\mathbb{R}^n$ are linear relations; that is,  linear subspaces of $\mathbb{R}^m\oplus\mathbb{R}^n$. The symmetric monoidal structure is given by direct sum.

Note that every each object $\R^n$ of $\LinRel$ is a monoidal product of $n$ copies of $\R$. Thus, to equip $\LinRel$ with a hypergraph structure, it suffices to give the object $\mathbb{R}$ a Frobenius structure. In fact, we can equip $(\LinRel, \oplus, 0)$ with \emph{two} different hypergraph category structures. Let's see how these are constructed.

Since the multiplication and comultiplication are maps $\mu\colon\mathbb{R}\otimes\mathbb{R}\to\mathbb{R}$ and $\delta\colon\mathbb{R}\to\mathbb{R}\oplus\mathbb{R}$, they will be defined by subspaces of $\mathbb{R}\oplus\mathbb{R}\oplus\mathbb{R}$; similarly, the unit and counit will correspond to subspaces of $\mathbb{R}$.

The first structure we consider has both $\mu$ and $\delta$ given by the
subspace $$\{ (v,v,v)\}\subseteq \mathbb{R}\oplus\mathbb{R}\oplus\mathbb{R}$$
and unit and counit given by the subspace $\mathbb{R}\subseteq\mathbb{R}$. The
second structure has multiplication defined as the subspace $$\{ (u,v,w)\mid
u+v=w\}\subseteq \mathbb{R}\oplus\mathbb{R}\oplus\mathbb{R},$$ comultiplication
given by $$\{ (u,v,w)\mid u=v+w\}\subseteq \mathbb{R}\oplus\mathbb{R}\oplus\mathbb{R}$$ and unit and counit $0\subseteq\mathbb{R}$. It's not hard to show that these satisfy the conditions needed to make $\mathbb{R}$ into a special Frobenius object, thus yielding two different hypergraph category structures on $\LinRel$.

More details can be found in \cite[Section 5.5 and 7.1]{BF}.
\end{ex}

\section{Decorating data}\label{section.decorating.data}
A convenient way to construct hypergraph categories and functors is by using decorated corelations. In this section we review the decorated cospans and decorated corelations constructions of \cite{Fong15,Fong18}, and we define a category $\Data$ of ``decorating data''. We will eventually see that $\Data$ contains all the relevant information needed to construct all hypergraph categories and hypergraph functors between them.

\subsection{Decorated cospans}

Cospan categories prove useful for representing {\it open networks} \cite{RSW05,Fong16}. However, sometimes it is necessary to record some extra information, such as labelings on the edges of a graph, or resistance values on an electrical circuit. For that, we can extend this structure to allow for what are called {\it decorations}.

\begin{defn}
Let $\C$ be a finitely cocomplete category, and consider a symmetric lax monoidal functor $(F,\varphi)\colon(\C,+)\to (\Set,\times)$. An \define{$F$-decorated cospan} in $\C$ is a pair
\[\left(
\btk[row sep=2ex]
& N &\\
X\ar[ur,"i"] & & Y\ar[ul,"o"']
\etk, \hspace{0.5cm}
\btk[row sep=2ex]
FN\\
1\ar[u,"s"']
\etk
\right)\]
where $X\xrightarrow{i} N \xleftarrow{o} Y$ is a cospan in $\C$ and $s$ is an element in the set $FN$. We say this is a cospan \define{decorated} by $s$.
\end{defn}

Two decorated cospans $(X\xrightarrow{i} N \xleftarrow{o} Y, s)$ and $(X\xrightarrow{i'} N' \xleftarrow{o'} Y', s')$ are isomorphic if there exists an isomorphism $f\colon N \to N'$ of cospans such that $Ff(s) =s'$.

Just like in the case of regular cospans, decorated cospans are the morphisms in some category. The following construction is the content of \cite[Prop.\ 3.2]{Fong15}.

\begin{defn}
Under the same hypotheses as above, we can define a category $\FCospan$ whose objects are the same as in $\C$, and whose maps $X\to Y$ are isomorphism classes of $F$-decorated cospans. Composition of decorated cospans $(X\xrightarrow{i} N \xleftarrow{o} Y, s)$ and $(Y\xrightarrow{i'} M\xleftarrow{o'} Z, t)$ is given by the usual composition of cospans,
\[
X\xrightarrow{j_Ni} N +_Y M \xleftarrow{j_Mo'} Z,
\]
together with the decoration
\[
1\cong 1\times 1\xrightarrow{s\times t} FN\times FM\xrightarrow{\varphi} F(N + M)\xrightarrow{F[j_N,j_M]} F(N +_Y M).
\]
\end{defn}

\begin{rmk}
We can see in the definition of $\FCospan$ why we need the decorations to be chosen through a lax monoidal functor; this structure is used to define composition of decorated cospans.
\end{rmk}

Trivially, we have an embedding $$\Cospan(\C)\hookrightarrow\FCospan$$ sending a cospan $X\xrightarrow{i} N \xleftarrow{o} Y$ to the same cospan with empty decoration
\[
\btk[column sep=small]
(X\rar["i"] & N & Y\lar["o"'], \ 1\rar["\varphi"] & F\emptyset\rar["F!"] & FN),
\etk
\]
where $!\colon \varnothing \to N$ is the unique such map. It is through this embedding that $\FCospan$ inherits a symmetric monoidal structure, and moreover, a hypergraph structure, from that of $\Cospan(\C)$ (see \cite[Thm. 3.4]{Fong15}).

\subsection{Decorated corelations}\label{corel}

Even though decorated cospans are useful for recording extra information present in open networks, they sometimes fail to be efficient, since they can carry redundant information that is inaccessible from the boundary. To solve this problem, {\it decorated corelations} were introduced in \cite{Fong18}. We now recall the definitions.

\begin{defn}
  A \define{factorisation system} $(\E,\M)$ in a category $\C$ consists of a pair of subcategories $\E,\M$ of $\C$ satisfying the following:
\begin{enumerate}[noitemsep, label=(\roman*)]
\item $\E$ and $\M$ contain all isomorphisms,
\item every morphism $f$ in $\C$ factors as $f=me$ for some $e\in\E$, $m\in\M$,
\item given factorisations $f=me, \ f'=m'e'$, for every $u,v$ such that $vf=f'u$, there exists a unique morphism $s$ making the diagram commute
    \bc\btk
    \bullet\rar["e"]\dar["u"'] & \bullet\rar["m"]\dar[dashed,"s"] & \bullet\dar["v"] \\
    \bullet\rar["e'"'] & \bullet\rar["m'"'] & \bullet
    \etk\ec
\end{enumerate}
In particular, note that the last condition implies that factorisations are unique up to unique isomorphism.
\end{defn}

When factoring a map $f\colon X\to N$ in a factorisation system, the notation we will use is
\[\btk[row sep=1.5ex]
& \overline{N}\ar[rd,"m"] & \\
X\ar[rr,"f"']\ar[ru,"e"] & & N
\etk\]

\begin{ex}\label{trivial_fact_sys}
Given any category $\C$, let $\mathcal{I}_\C$ denote the subcategory containing all objects of $\C$ and isomorphisms between them. Then, there are two factorisation systems one can always consider: $(\C,\mathcal{I}_\C)$ and $(\mathcal{I}_\C,\C)$.

When $\C$ is clear from context, we will denote $\mathcal{I}_\C$ simply by $\mathcal{I}$.
\end{ex}

\begin{ex}
A commonly used factorisation system in $\Set$ is the pair $(\E,\M)$ where $\E$ consists of all surjections and $\M$ of all injections.
\end{ex}

\begin{ex}
Along the same lines as the previous example, any abelian category admits a factorisation system $(\E,\M)$ in whch $\E$ is the subcategory of all epimorphisms and $\M$ consists of all monomorphisms.
\end{ex}

\begin{defn}
Let $\C$ be a finitely cocomplete category and $(\E,\M)$ a factorisation system in $\C$. An $(\E,\M)$-corelation is a cospan $X\xrightarrow{i} N \xleftarrow{o} Y$ such that the universal map from the coproduct of the feet to the apex, displayed below, belongs to the subcategory $\E$
\[\btk[row sep=2ex]
& N & \\
& X+Y\ar[u,dashed, "{[i,o]}"] & \\
X\ar[uur,"i",bend left]\ar[ur,"\iota_X"'] & & Y\ar[uul,"o"',bend right]\ar[ul,"\iota_Y"]
\etk\]
\end{defn}

\begin{ex}\label{trivial_corels}
For a simple yet illustrative example, one can quickly check that $(\C,\mathcal{I}_\C)$-corelations are just cospans in $\C$, and that there exists a unique $(\mathcal{I}_\C,\C)$-corelation from $X$ to $Y$; namely, $X\xrightarrow{\iota_X} X+Y\xleftarrow{\iota_Y} Y$.
\end{ex}

If we require $(\E,\M)$ to satisfy an additional property that will allow composition to work properly, these structures can be assembled into the category described in the following definition, as shown in \cite[Thm.~3.1]{Fong18}.

\begin{defn}\label{corels}
Let $\C$ be a finitely cocomplete category and $(\E,\M)$ a factorisation system in $\C$ such that $\M$ is stable under pushout. We can define a category $\Corel_{(\E,\M)}(\C)$ with the same objects as $\C$ and with isomorphism classes of $(\E,\M)$-corelations as morphisms.

Corelations $X\xrightarrow{i} N \xleftarrow{o} Y$ and $Y\xrightarrow{i'} M \xleftarrow{o'} Z$ compose by taking the pushout as usual, and then factoring the induced map $X+Z\to N+_Y M$ via the factorisation system $(\E,\M)$ as shown in the diagram
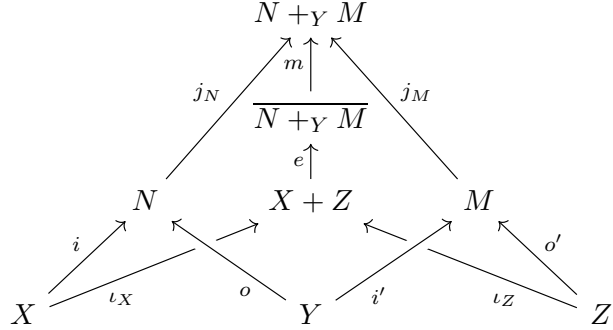
\[\btk
& & N+_Y M & & \\[1pt]
& & \overline{N+_Y M}\ar[u,"m"] & & \\[-6pt]
& N\ar[ruu,"j_N"] & X+Z\ar[u,"e"] & M\ar[luu,"j_M"'] & \\[1.5ex]
X\ar[ru,"i"]\ar[rru,"\iota_X"', near start] & & Y\ar[lu,"o", near start, crossing over] & & Z\ar[lu,"o'"']\ar[llu,"\iota_Z", near start] \ar[from=4-3, to=3-4,"i'"',near start, crossing over]
\etk\] The composite is then the corelation $X \xrightarrow{e\iota_X} \overline{N+_Y M} \xleftarrow{e\iota_Z} Z$.

We will often denote the category $\Corel_{(\E,\M)}(\C)$ simply by $\Corel(\C)$, if the factorisation system is clear from context.
\end{defn}

There exists a functor $$\Cospan(\C)\to\Corel(\C)$$ taking a cospan to its $\E$-part; that is, the cospan $X\xrightarrow{i} N \xleftarrow{o} Y$ is associated the corelation $X\xrightarrow{e\iota_X} \overline{N} \xleftarrow{e\iota_Y} Y$, where we have the factorisation $[i,o] = me$.
It is possible to give $\Corel(\C)$ the structure of a hypergraph category so that this functor is a hypergraph functor.

Just like in the case of cospans, we can talk about a notion of decoration when dealing with corelations, that allows us to keep track of extra information.

First, we define a category of ``restricted cospans'', where we specify the subcategories on which the left and right legs of the cospans are allowed to take values.

\begin{defn}
For $\C$ a finitely cocomplete category, and $\M\subseteq\C$ a subcategory stable under pushouts, we can define the category $\C\cp\M^\op$ whose objects are the same as in $\C$ and whose maps $X\to Y$ are isomorphism classes of cospans $X\xrightarrow{f} N \xleftarrow{m} Y$ with $f$ in $\C$ and $m$ in $\M$.

This inherits a symmetric monoidal category structure from the coproduct in $\C$, and as such, is a monoidal subcategory of $\Cospan(\C)$.
\end{defn}

\begin{ex}
For any finitely cocomplete $\C$, if we set $\M=\C$ then $\C\cp\M^\op$ is precisely $\Cospan(\C)$.
\end{ex}

\begin{defn}
Let $\C$ be a finitely cocomplete category, $\M\subseteq\C$ a subcategory stable under pushouts which is part of a factorisation system $(\E,\M)$ in $\C$, and consider a lax symmetric monoidal functor
\[
(F,\varphi)\colon(\C\cp\M^\op ,+)\to(\Set,\times).
\]

An \define{$F$-decorated corelation} is a pair
\[\left(
  \btk[row sep=2ex]
& N &\\
X\ar[ur,"i"] & & Y\ar[ul,"o"']
\etk, \hspace{0.5cm}
\btk[row sep=2ex]
FN\\
1\ar[u,"s"']
\etk
\right)\]
where $X\xrightarrow{i} N \xleftarrow{o} Y$ is an $(\E,\M)$-corelation.
\end{defn}

Decorated corelations are the morphisms in the following category, constructed in \cite[Thm. 5.8]{Fong18}.

\begin{defn}\label{FCorel}
Under the same hypotheses as above, we can define a hypergraph category $\FCorel$ whose objects are the same as in $\C$, and whose maps $X\to Y$ are isomorphism classes of decorated corelations. Composition of decorated corelations $(X\xrightarrow{i} N \xleftarrow{o} Y, 1\xrightarrow{s} FN)$ and $(Y\xrightarrow{i'} M\xleftarrow{o'} Z, 1\xrightarrow{t} FN)$ is given by the composition of corelations defined in Defn.\ \ref{corels}, $$X \xrightarrow{e\iota_X} \overline{N+_Y M} \xleftarrow{e\iota_Z} Z$$ together with the decoration
$$1\cong 1\times 1\xrightarrow{s\times t} FN\times FM\xrightarrow{\varphi} F(N + M)\xrightarrow{F[j_N,j_M]} F(N +_Y M)\xrightarrow{Fm^\op} F(\overline{N+_Y M}).$$
\end{defn}

\begin{rmk}
Once again, the definition of $\FCospan$ makes it clear why we need the decorations to be chosen through a lax monoidal functor, in this case, not from $\C$ but from $\C\cp \M^\op$: it provides the maps we need to compose decorations.

More explicitly, this is what provides us with the map $$F(N +_Y M)\xrightarrow{Fm^\op } F(\overline{N+_Y M})$$ used above; when composing the corelations we get a map $\overline{N+_Y M}\xrightarrow{m} N+_Y M$ (for reference, see the diagram in Defn.\ \ref{corels}), from which we can build the cospan $m^\op \coloneqq (\btk[cramped]N+_Y M\ar[r,equal] & N+_Y M & \overline{N+_Y M}\ar[l,"m"']\etk)$.
This is a morphism in $\C\cp\M^\op$ from $N+_Y M$ to $\overline{N+_Y M}$; we then apply $F$ to get the map $F(N+_Y M)\xrightarrow{Fm^\op } F(\overline{N+_Y M})$ that we need.
\end{rmk}

Let's elaborate on Example \ref{trivial_corels}, and study decorated corelations on the factorisation systems $(\mathcal{I},\C)$ and $(\C,\mathcal{I})$.

\begin{ex}\label{trivial_dec_corels1}
For the pair $(\mathcal{I},\C)$, we have $\C\cp \C^\op =\Cospan(\C)$, and thus decorated $(\mathcal{I},\C)$-corelations will be given through a lax monoidal functor $F\colon\Cospan(\C)\to\Set$.

For each pair of objects $X, Y$ in $\C$, a morphism in $\FCorel_{(\mathcal{I},\C)}$ from $X$ to $Y$ is a pair $$(X\xrightarrow{\iota_X} X+Y\xleftarrow{\iota_Y} Y, \, 1\xrightarrow{s} F(X+Y))$$ consisting of the unique $(\mathcal{I},\C)$-corelation from $X$ to $Y$ together with a decoration on the apex. We can thus ignore this first coordinate, and set $$ \FCorel_{(\mathcal{I},\C)}(X,Y)=F(X+Y).$$

Composition in this category is given as follows. Write $\comp$ for the cospan
\[
  X+Y+Y+Z \xrightarrow{\id_X+[\id_Y,\id_Y]+\id_Z} X+Y+Z \xleftarrow{[\iota_X,\iota_Z]} X+Z.
\]
Then composition is given by
\[
  F\comp\circ \varphi\colon F(X+Y) \times F(Y+Z) \to F(X+Z).
\]
\end{ex}

\begin{ex}\label{trivial_dec_corels2}
If instead we consider the factorisation system $(\C,\mathcal{I})$, we see that $\C\cp\mathcal{I}^\op \cong\C$ and so decorated $(\C,\mathcal{I})$-corelations will be given through a lax monoidal functor $F\colon\C\to\Set$.

Morphisms in $\FCorel_{(\C,\mathcal{I})}$ from $X$ to $Y$ are pairs $$(X\xrightarrow{i} N \xleftarrow{o} Y, \, 1 \xrightarrow{s} F(N))$$ where the induced map $X+Y\xrightarrow{[i,o]} N$ is in $\C$; this imposes no extra conditions, and so for this factorisation system we have $\FCorel_{(\C,\mathcal{I})}=\FCospan$.
\end{ex}

Observe that we have an embedding $\Corel(\C)\hookrightarrow\FCorel$ that sends a corelation $X\xrightarrow{i} N \xleftarrow{o} Y$ to the same corelation with empty decoration
\[\btk[column sep=small]
(X\rar["i"] & N & Y\lar["o"'], \ 1\rar["\varphi"] & F\emptyset\rar["F!"] & FN).
\etk\]
Analogously to decorated cospans, the hypergraph structure on $\FCorel$ is inherited from that on $\Corel(\C)$.

In this paper we show this construction is functorial. But first we have to define a category of decorating data.

\subsection{The category of decorating data}

At this point, we have two ways of constructing decorated categories: decorated
cospans, and decorated corelations. But decorated cospan categories are not
essential to the picture: in Example \ref{trivial_dec_corels2} we saw that
$\FCospan= \FCorel_{(\C,\I)}$. Thus it is enough just to think about decorated
corelations, which stem from lax monoidal functors $$F\colon(\C\cp \M^\op
,+)\to(\Set,\times)$$ for some category $\C$ and some factorisation system
$(\E,\M)$ in $\C$. These will be the components of our decorating data category.

\begin{defn}
Let $\FactSys$ denote the category whose objects are pairs $(\C,(\E,\M))$, where $\C$ is a finitely cocomplete category and $(\E,\M)$ a factorisation system in $\C$, and whose morphisms $(\C,(\E,\M))\to(\C',(\E',\M'))$ are given by finite colimit-preserving functors $A\colon C\to\C'$ such that $A(\M)\subseteq\M'$.

There exists a functor $$D\colon\FactSys^\op \to \Cat$$ taking a pair $(\C,(\E,\M))$ to the category $\Lax((\C\cp \M^\op ,+),(\Set,\times))$ of lax symmetric monoidal functors and monoidal natural transformations, and a map $A\colon(\C,(\E,\M))\to(\C',(\E',\M'))$ to $DA=-\circ A$; that is, precomposition with the functor $A$.
\end{defn}
\begin{defn}
  We introduce the category of decorating data $$\Data\coloneqq\int^{\FactSys^\op } D,$$ that is, the Grothendieck construction on the functor $D$. Recall that the objects of this category are tuples $$(\C,(\E,\M), F), \qquad \mbox{where }(\C,(\E,\M)) \in \FactSys,\quad F \in \Lax(\C\cp \M^\op ,\Set),$$
and the morphisms $(\C,(\E,\M), F)\to (\C',(\E',\M'), F')$ consist of pairs $(A,\alpha)$, where $A\colon\C\to\C'$ is a finitely cocontinuous functor satisfying $A(\M)\subseteq\M'$, and $\alpha\colon F\Rightarrow DA(F')$ is a monoidal natural transformation
\[\btk
\C\cp\M^\op \rar["F"]\dar["A"',"\ \ \ \Downarrow\alpha" near start] & \Set\\
\C'\cp {\M'}^\op \ar[ur,"F'"'] &
\etk\]
\end{defn}

\subsection{Decorating corelations is functorial}

We now proceed to assemble the decorated corelation construction, explained in Definition \ref{FCorel}, into a functor $\Data \to \Hyp$. We begin by recalling a way to obtain hypergraph functors between decorated corelation categories.

\begin{prop} \label{prop.deccorelfunctor}
Given a morphism $$(A,\alpha)\colon(\C, (\E,\M), F)\to(\C', (\E',\M'), F')$$ in $\Data$,
  we may define a hypergraph functor
  \[
    (A,\alpha)\Corel\colon \FCorel\longrightarrow F'\Corel
  \]
  as follows:
\begin{itemize}
  \item on objects, it takes $X\in\C$ to $AX\in\C'$;
  \item on morphisms, given an $F$-decorated corelation
    \[
      (X\xrightarrow{i} N\xleftarrow{o} Y, \, 1 \xrightarrow{s} FN)
    \]
    we factor the map $AX+AY\cong A(X+Y)\xrightarrow{A[i,o]} AN$ in $(\E',\M')$ as $A(X+Y)\xrightarrow{e'} \overline{AN}\xrightarrow{m'} AN$, and then consider the corelation
    \[
      AX\xrightarrow{e'\iota_{AX}} \overline{AN} \xleftarrow{e'\iota_{AY}} AY.
    \]
    For the decoration, we precompose
    $F'({m'}^\op )\colon F'AN\to F'\overline{AN}$ with $\alpha_N \colon FN\to F'AN$ to obtain an element $$t=\big(F'({m'}^\op )\circ \alpha_N \big)(s)\in F'\overline{AN}.$$
\end{itemize}
\end{prop}
\begin{proof}
This can be found in \cite[Prop.\ 6.1]{Fong18}.
\end{proof}

\begin{thm} \label{thm.deccorelfunctorial}
There exists a functor
\[
  (-)\Corel\colon\Data\longrightarrow\Hyp
\]
which, on objects, takes decorating data  $(\C, (\E,\M), F)$ to the hypergraph category $\FCorel$, and whose action on morphisms is described in Prop.\ \ref{prop.deccorelfunctor}.
\end{thm}

We will not prove this just yet. It is straightforward, but lengthy, to give an elementary proof: the map respects composition of corelations and composition of decorations due to the universal property of the factorisation systems. We, however, choose to defer this proof to Section \ref{section.corel.functor}, where we show that this correspondence can be expressed as the composition of two functors, and is therefore a functor itself.


\section{The category $\calg$}\label{section.cospanalg}

In this section, we single out a special subcategory
$$\calg\hookrightarrow\Data.$$ This embedding turns out to be part of an
adjunction, which makes $\calg$ a reflective subcategory of $\Data$, with the left adjoint being a Kan extension functor.

\subsection{The subcategory $\calg$}

If we consider any finitely cocomplete category $\C$ as equipped with the trivial factorisation system $(\I,\C)$, we may view the category of finitely cocomplete categories and finite colimit-preserving functors as a subcategory of $\FactSys$. Restricting our Grothendieck construction to this subcategory gives the subcategory $\calg\subseteq\Data$.

\begin{defn}
Write $\calg$ for the full subcategory of $\Data$ whose objects are the tuples $(\C, (\I,\C), F)$, with trivial factorisation system $(\I,\C)$.
\end{defn}

More explicity still, $\calg$ is the category whose objects are simply pairs $(\C,F)$, where $\C$ is a finitely cocomplete category, and $F\colon \Cospan(\C)\to\Set$ is a lax monoidal functor.

By definition, we have an embedding $$\iota\colon \calg \hookrightarrow\Data.$$ It is also possible to construct a functor in the other direction, which we call $\Kan$ since it consists of taking the left Kan extension of the lax monoidal functors present in $\Data$.

\begin{prop} \label{def.kan}
There exists a functor $$\Kan\colon\Data\to\calg$$ taking decorating data $(\C, (\E,\M), F)$ to the cospan algebra $(\C,\Lan F)$, where $\Lan F$ is the left Kan extension of $F$ along the inclusion $\C\cp \M^\op \hookrightarrow\Cospan(\C)$:
\[\btk
\C\cp\M^\op \rar["F"]\dar[hookrightarrow," \ \ \ \Downarrow\kappa" near start] & \Set\\
\Cospan(\C)\ar[ru,"\Lan F"']
\etk\]

For a morphism $(A,\alpha)\colon (\C, \, (\E,\M), \, F)\to (\C', \, (\E',\M'), \, F')$, the associated map of cospan algebras
\[
\Kan(A,\alpha)\colon (\C,\, \Lan F)\to(\C', \, \Lan F')
\]
is the pair $(A,\beta)$, where $\beta\colon \Lan F\Rightarrow\Lan F' A$ is the natural transformation given by the universal property of $(\Lan F,\kappa)$.
\end{prop}

\begin{proof}
In order for this definition to be correct, we must make sure that the image of $\Kan$ lands in $\calg$ as claimed; that is, that $(\C, \Lan F)$ is a cospan algebra and the map $(A,\beta)\colon (\C,\Lan F)\to (\C',\Lan F')$ defined above is a morphism of cospan algebras.

This amounts to proving that the functor $\Lan F$ is lax monoidal, and that the natural transformation $\beta$ given by the universal property is also monoidal, which is ensured by Proposition 1.1 (a) and (c) in \cite{FP}. Checking that our setting fits the hypotheses of this result is mostly a matter of checking $\Lan F \times \Lan F$ is a Kan extension of $F \times F$; this is straightforward, and can be done analogously to our proof of Proposition \ref{prop.kanexplicit}.

Finally, the functoriality of $\Kan$ is a routine check.
\end{proof}

We end this section with the following result, showing that the two functors we just defined are adjoints.

\begin{prop}
The functor $\Kan$ is left adjoint to $\iota$.
\[\btk
\Data\rar[shift left=5pt, "\Kan" pos=0.45,""{name=a}] & \calg\lar[shift left=5pt,"\iota" pos=0.45,""{name=b}] \ar[phantom, from=a, to=b,"\scriptstyle\bot"]
\etk\]
\end{prop}
\begin{proof}
  Recall that left Kan extension is left adjoint to precomposition giving, for each $\C\cp\M^\op \hookrightarrow \Cospan(\C)$,
  \begin{equation} \label{eq.lanadj}
    \btk
\Lax(\C\cp\M^\op,\Set) \rar[shift left=5pt, "\Lan" pos=0.45,""{name=a}] &
\Lax(\Cospan(\C),\Set) \lar[shift left=5pt,""{name=b}]
\ar[phantom, from=a, to=b,"\scriptstyle\bot"]
\etk
\end{equation}
Thus for the unit and counit of the adjunction $\Kan \vdash \iota$, we may use those of the adjunctions \eqref{eq.lanadj}. These then immediately obey the triangle equations, the unit is natural because precomposition with $\C\cp\M^\op \hookrightarrow \Cospan(\C)$ is natural in $\FactSys$ (and the universal property of left Kan extension), and the counit is natural because it is in fact the identity.
%
%
\end{proof}

In particular, the above result implies that $\calg$ is a reflective subcategory of $\Data$.

\subsection{An explicit formula for $\Kan$}

In Proposition \ref{def.kan} above, the functor $\Kan\colon \Data\to\calg$ was defined rather abstractly, making use of the universal property of left Kan extensions. We wish to give a more explicit description of this functor, that will allow us to compare it to other constructions and get a good sense of how this fits into our machinery for building hypergraph categories.

Recall, for example from \cite[Thm.~X.3.1]{Mac}, that in a finitely cocomplete category left Kan extensions can be computed by the formula $$\Lan F(-)=\colim(\iota\mathord{\downarrow}(-)\xrightarrow{\text{proj}} \C\cp\M^\op \xrightarrow{F} \Set).$$  A careful inspection of this colimit yields the following two propositions; elementary proofs are also provided in Appendix~\ref{appendix}.

\begin{prop}[$\Kan$ on objects] \label{prop.kanexplicit}
  Let $(\C,(\E,\M),F)$ be an object on $\Data$. Then the left Kan extension of $F$ along the embedding $\C\cp \M^\op \hookrightarrow\Cospan(\C)$ is the functor $\Lan F\colon \Cospan(\C)\to\Set$ which takes an object $X$ in $\C$ to $$\Lan F X=\{(X\xrightarrow{e} N, 1\xrightarrow{s}FN)\},$$ where $e$ represents an isomorphism class of objects in $X\mathord\downarrow\E$.

  Given a morphism $f=(X\xrightarrow{i} M\xleftarrow{o} Y)$ in $\Cospan(\C)$ and an element $(X\xrightarrow{e} N, 1\xrightarrow{s} FN)$ in $\Lan FX$, we compose $e^\op$ with $f$ to get a cospan $N\xrightarrow{j_N} N+_X M\xleftarrow{j_Mo} Y$. Factoring the right leg of this cospan in $(\E,\M)$, we obtain maps $$Y\xrightarrow{e'} \overline{N+_X M} \xrightarrow{m} N+_X M.$$
We then define
\begin{align*}
  \Lan F (f) \colon \Lan F X &\longrightarrow \Lan F Y \\
(X\xrightarrow{e} N,\, s) &\longmapsto (Y\xrightarrow{e'} \overline{N+_X M},\, F(m^\op )F(j_N)(s)).
\end{align*}

Furthermore, the associated natural transformation
\[\btk
\C\cp\M^\op \dar[hookrightarrow, "\ \ \ \Downarrow\kappa" near start ]\rar["F"] & \Set \\
\Cospan(\C)\ar[ru,"\Lan F"']
\etk\]
is defined on each component by
\begin{align*}
\kappa_X \colon FX&\longrightarrow\Lan FX \\
s&\longmapsto \btk[column sep=small,ampersand replacement=\&] (X\rar[equal] \& X, s). \etk
\end{align*}
\end{prop}

\begin{prop}[$\Kan$ on morphisms]\label{Kan.on.maps}
Let
\[
  (A,\alpha) \colon (\C, \, (\E,\M), \, F)\to (\C', \, (\E',\M'), \, F')
\]
be a morphism on $\Data$. Then, the map of cospan algebras
\[
\Kan(A,\alpha) \colon (\C,\, \Lan F)\to(\C', \, \Lan F')
\]
is the pair $(A,\beta)$. Here $\beta\colon \Lan F\Rightarrow \Lan F' A$ is the natural transformation whose components $\beta_X \colon \Lan F X\to \Lan F' AX$ are given by $$(X\xrightarrow{e} N, s)\mapsto (AX\xrightarrow{e'} \overline{AN}, F'({m'}^\op )\alpha_N(s)),$$ where we have the $(\E',\M')$-factorisation $Ae = m' \circ e'$.
\end{prop}

\section{Deconstructing decorated corelations}\label{section.hyp.and.cospanalg}

In this section we prove our main result, that the decorated corelations construction factors as $\Kan$ then $\algtohyp$, and hence is functorial. 

\subsection{Decorating using cospan algebras}

When restricting the decorated cospan constructions to the category $\calg$, where the factorization systems considered are trivial, verifying that the correspondence is functorial becomes a simple check.

\begin{lemma} \label{lem.deccospanalgfunctorial}
  Write $\algtohyp$ for the restriction of the map
  $
  (-)\Corel\colon \Data \to \Hyp
  $
  to the subcategory $\calg \subseteq \Data$. The map
  \[
    \algtohyp\colon \calg \to \Hyp
  \]
  is functorial.
\end{lemma}
\begin{proof}
  It is immediate that $\algtohyp$ preserves identities. Suppose we have the pair of composable morphisms
  \[
    \btk[column sep=large]
    \Cospan(\C) \ar[rd,"F"] \ar[d, "A"',"\ \ \ \ \Downarrow\alpha" near end] \\
    \Cospan(\C')\ar[r,"F'" description] \ar[d, "B"',"\ \ \ \ \Downarrow\beta" near start] & \Set \\
    \Cospan(\C'') \ar[ru,"F''"']
    \etk
  \]
  in $\calg$. We must show that
  \[
    (B,\beta)\Corel \circ (A,\alpha)\Corel = (BA,\beta\circ\alpha)\Corel.
  \]
  On objects this is easy: both functors map each object $X$ of $F\Corel$ to $BAX$ in $F''\Corel$.

  For morphisms, recall from Example~\ref{trivial_dec_corels1} that since $F$ is in $\calg$, a morphism $X \to Y$ in $F\Corel$ is simply an element $s \in F(X+Y)$. By Proposition~\ref{prop.deccorelfunctor}, the images of this morphism $s$ under the functors $(B,\beta)\Corel \circ (A,\alpha)\Corel$ and $(BA,\beta\circ\alpha)\Corel$ are given respectively by the upper and lower composites in the diagram
  \[\btk[row sep=3ex,column sep=2.5ex,font=\small]
  1\ar[r,"s"]
  & F(X+Y)\ar[r,"\alpha_{X+Y}"]
  & F'A(X+Y)\ar[r,"F'(\sim)"]\ar[d,"\beta_{A(X+Y)}"']
  &[3ex] F'(AX+AY)\ar[d,"\beta_{AX+AY}"]
  \\
  & & F''BA(X+Y)\ar[r,"F''B(\sim)"]
  & F''B(AX+AY)\ar[r,"F''(\sim)"]
  &[3ex] F''(BAX+BAY)
  \etk\]
  where $\sim$ are the relevant maps given by the universal property of the coproduct, which are isomorphisms since $A$ and $B$ are cocontinuous. The square commutes by the naturality of $\beta$; this proves the two required functors are equal.
\end{proof}

\subsection{Decorated corelations is $\Kan$ then $\algtohyp$}\label{section.corel.functor}

First, we must understand how the functor $\Kan$ interacts with decorating corelations.  We show that $(-)\Corel$ factors through $\calg$.

\begin{thm}\label{thm.kan}
  The diagram
\[
\btk[row sep=1.5ex,column sep=6ex]
  \calg \ar[rd,"\algtohyp" near start]\\
  & \Hyp \\
  \Data \ar[uu,"\Kan"] \ar[ru,"(-)\Corel"' near start]
  \etk
\]
commutes in $\Cat$.
\end{thm}

To prove this, we show the following two lemmas, which separately study the action of these correspondences on objects and on maps.

\begin{lemma} \label{lem.corelkanobjects}
  Let $(\C,(\E,\M),F)$ be an object in $\Data$. Then
  \[
    \FCorel_{(\E,\M)} = (\Lan F)\Corel_{(\I,\C)}
  \]
  as hypergraph categories.
\end{lemma}

\begin{proof}
  Note that, by Proposition~\ref{prop.deccorelfunctor}, the pair $(\id_\C,\kappa)$, where $\kappa$ is the canonical natural transformation $F \to \Lan F \circ \iota$ of the left Kan extension, induces a hypergraph functor $I\coloneqq(\id_\C,\kappa)\Corel\colon \FCorel_{(\E,\M)} \to (\Lan F)\Corel_{(\I,\C)}$. Simply by unpacking definitions, we will show that $I$ is identity-on-objects and identity-on-morphisms.

  The object case is trivial: both $\FCorel_{(\E,\M)}$ and $(\Lan F)\Corel_{(\I,\C)}$ have the same objects as $\C$, and $\id_\C$ is the identity functor, so $I$ is identity-on-objects.

  Let's now consider morphisms. Fix objects $X,Y$ in $\C$. The homset $\FCorel_{(\E,\M)}(X,Y)$ is the set of isomorphism classes of $F$-decorated $(\E,\M)$-corelations $(X\xrightarrow{i} N\xleftarrow{o} Y, \ s\in FN)$.
On the other hand, Example~\ref{trivial_dec_corels1} shows that the homset $(\Lan F)\Corel_{(\I,\C)}(X,Y)=\Lan F(X+Y)$, and by Proposition~\ref{prop.kanexplicit}, this set is just the set of isomorphism classes of pairs $(X+Y\xrightarrow{e} N, \ s\in FN)$. By a very minor abuse of notation, we consider $(\E,\M)$-corelations $X\xrightarrow{i} N\xleftarrow{o} Y$ to be the same as maps $X+Y \xrightarrow{e} N$, and hence consider $\FCorel_{(\E,\M)}(X,Y)$ equal to $(\Lan F)\Corel_{(\I,\C)}(X,Y)$. It remains to check that $I$ is the identity on this set.

  Let $(X\xrightarrow{i} N\xleftarrow{o} Y, \ s\in FN)$ be an $F$-decorated $(\E,\M)$-corelation. By Proposition~\ref{prop.deccorelfunctor}, its image is $\big(\Lan F([i,o]^\op )\big)\circ \kappa (s)$. But by Proposition~\ref{prop.kanexplicit}, this is exactly the pair $(X+Y \xrightarrow{[i,o]} N,  s\in FN)$, which we consider to the same as the corelation we started with. Thus $I$ is identity-on-morphisms, and thus $\FCorel_{(\E,\M)} = (\Lan F)\Corel_{(\I,\C)}$ as hypergraph categories.
\end{proof}

\begin{lemma} \label{lem.corelkanmorphisms}
  Let $(A,\alpha)\colon (\C,(\E,\M),F)\to (\C',(\E',\M'),F')$ be a morphism in $\Data$. Then
  \[
    (A,\alpha)\Corel = (\Kan(A,\alpha))\Corel
  \]
  as hypergraph functors $F\Corel \to F'\Corel$.
\end{lemma}

\begin{proof}
 This is precisely the statement of Prop.\ \ref{Kan.on.maps}, when compared to the definition of the functor $(-)\Corel$ given in Thm.\ \ref{thm.deccorelfunctorial}.
\end{proof}
\begin{proof}[Proof of Theorems~\ref{thm.deccorelfunctorial} and \ref{thm.kan}]
  By Lemmas~\ref{lem.corelkanobjects} and \ref{lem.corelkanmorphisms}, we see respectively that on objects and on morphisms we have $(-)\Corel = \algtohyp \circ \Kan$. Since $\Kan$ is functorial by Definition~\ref{def.kan}, and $\algtohyp\colon \calg \to \Hyp$ is functorial by Lemma~\ref{lem.deccospanalgfunctorial}, this implies that $(-)\Corel\colon \Data \to \Hyp$ is functorial. This proves Theorem~\ref{thm.deccorelfunctorial}.

  Moreover, now that we know $(-)\Corel$ is a functor, Lemmas~\ref{lem.corelkanobjects} and \ref{lem.corelkanmorphisms} prove the diagram commutes in $\Cat$, so we also have Theorem~\ref{thm.kan}.
\end{proof}

\section{All hypergraph categories are decorated corelation categories}\label{recover_hyp}

We devote this section to showing that every hypergraph category can be expressed, up to equivalence, as a decorated corelation category built from some decorating data. This supports our claim that $\Data$ is a suitable setting in which to work when constructing black-box functors. 

Furthermore, if the hypergraph category is \emph{object-wise free}, we can recover it as a decorated corelation category up to isomorphism.

\subsection{From hypergraph categories to cospan algebras}
Given a hypergraph category $\H$, we can construct a cospan algebra in two steps. First, we make use of the fact that $\Cospan(\FinSet)$ is the free hypergraph category, and use this to construct a functor from $\Ob\H$-many copies of $\Cospan(\FinSet)$ to $\H$, that describes the Frobenius structure in $\H$. Second, we then use the hom functor on the monoidal unit $\H(I,-)\colon \H \to \Set$ to capture the homsets of $\H$. As we shall see, this is enough to construct a cospan algebra that captures all the structure of $\H$.

\begin{defn}
  Let $\Lambda$ be a set. The comma category $\FinSet\comma\Lambda\coloneqq \FinSet\mathord{\downarrow}\Lambda$ is the one whose objects are functions $x \colon n\to\Lambda$ for some finite set $n$, and whose morphisms $(n\xrightarrow{x}\Lambda)\to(m\xrightarrow{y}\Lambda)$ are functions $f \colon n\to m$ such that the diagram below commutes
\[\btk[column sep=small]
n\ar[rr,"f"]\ar[dr,"x"'] & & m\ar[dl,"y"]\\
& \Lambda &
\etk\]

Alternatively, objects of $\FinSet\comma\Lambda$ can be interpreted as finite lists of elements in $\Lambda$. We call the objects of $\FinSet\comma\Lambda$ \define{labelled finite sets}, and $\Lambda$ the set of labels.
\end{defn}

\begin{rmk}
  Using colimits in $\Set$, it is easy to show that $\FinSet\comma\Lambda$ is finitely cocomplete. In fact, it will be important that $\FinSet\comma\Lambda$ is the free finitely cocomplete category on $\Lambda$ \cite[Ch.\ 6]{Joh77}.
\end{rmk}

\begin{prop}\label{frob}
  Let $\H$ be a hypergraph category.  There exists an identity-on-objects hypergraph functor $$\Frob \colon \Cospan(\FinSet\comma{\Ob\H})\to\H$$ whose image is the subcategory of $\H$ generated by the Frobenius morphisms of the hypergraph structure.
\end{prop}
\begin{proof}
  Theorem~3.14 of \cite{FS} states that given a set $\Lambda$, $\Cospan(\FinSet \comma \Lambda)$ is the free hypergraph category on the set $\Lambda$. The functor $\Frob$ is the map given by this universal property. More concretely, it is the functor generated by sending, for every $X \in \Ob\H$, the morphism $(X,X) \to X \leftarrow X$ in $\Cospan(\FinSet\comma{\Ob\H})$ to $\mu_X$, $\varnothing \to X \leftarrow X$ to $\eta_X$, $X \to X \leftarrow (X,X)$ to $\delta_X$, and $X \to X \leftarrow \varnothing$ to $\epsilon_X$.
\end{proof}

\begin{prop}
  We can construct a functor $$\hyptoalg \colon \Hyp\to\calg$$ sending a hypergraph category $\H$ to the pair $(\FinSet\comma{\Ob\H}, \, A_\H)$, where
  \[
    A_\H \colon \Cospan(\FinSet\comma{\Ob\H})\to\Set
  \]
  is defined by $A_\H(-)\coloneqq\H(I,\Frob(-))$.

  On morphisms, $\hyptoalg$ maps a hypergraph functor $F \colon \H\to\K$ to the pair $(A_F,\alpha)$, where $A_F \colon \FinSet\comma{\Ob\H}\to\FinSet\comma{\Ob\K}$ is the functor taking a list $(X_1,\dots,X_n)$ in $\H$ to the list $(FX_1,\dots FX_n)$ in $\K$ (and is identity on morphisms), and $$\alpha \colon \H(I,\Frob(-))\Rightarrow\K(I,\Frob A_F(-))$$ is the natural transformation given by \begin{align*}
\alpha_{(X_1,\dots,X_n)} \colon \H(I,X_1\otimes\dots\otimes X_n)& \to\K(I,FX_1\otimes\dots\otimes FX_n))\\
(I\xrightarrow{s} X_1\otimes\dots\otimes X_n) & \mapsto (I\xrightarrow{\varphi_I} FI\xrightarrow{\varphi^{-1}Fs} FX_1\otimes\dots\otimes FX_n),
\end{align*}
where $\varphi_I$ and $\varphi$ denote structure maps of the (strong) monoidal functor $F$.
\end{prop}
\begin{proof}
 First, we note that $\hyptoalg(\H)$ is indeed an object of $\calg$, since  the functor $A_\H$ is the composite of the monoidal functor $\Frob$ constructed in Proposition \ref{frob}, with the hom functor on the monoidal unit (which is lax monoidal), and is therefore lax monoidal.

  Moreover, it is easy to see that $A_F$ preserves finite colimits, and that $\alpha$ is a monoidal natural transformation (which follows from the hypergraph functor structure of $F$); this implies $(A_F,\alpha)$ is a morphism in $\calg$.

  These two facts show that $\hyptoalg$ is well defined. The functoriality of $\hyptoalg$ is a routine check.
\end{proof}

\subsection{$\Hyp$ is a coreflective subcategory of $\calg$}

We now prove that the functors $\algtohyp:\calg\to\Hyp$ and $\hyptoalg\colon\Hyp\to\calg$ are adjoint. We shall need the following technical lemma.

\begin{lemma}\label{mon}
Let $\C$ be a small finitely cocomplete category. There exists a functor $$\mon\colon \FinSet\comma{\Ob\C}\to\C$$ taking a list $(X_1,\dots,X_n)$ of objects of $\C$ to the object $X_1 + \dots + X_n$, and a map $f\colon (n\xrightarrow{x}\Lambda)\to(m\xrightarrow{y}\Lambda)$ to the morphism $X_1 + \dots + X_n\to Y_1 + \dots + Y_m$ in $\C$ induced by the maps $$X_i\xrightarrow{\id} X_i=Y_{f(i)}\xrightarrow{\iota_{Y_{f(i)}}} Y_1 + \dots + Y_m.$$
Furthermore, this functor is finitely cocontinuous.
\end{lemma}
\begin{proof}
Functoriality of $\mon$ is evident, and it is straightforward to prove, for example, that $\mon$ preserves finite coproducts and coequalizers.
\end{proof}

\begin{thm}\label{A.and.phi}
The functor $\hyptoalg$ is left adjoint to $\algtohyp$.
\[\btk
\Hyp\rar[shift left=5pt, "\hyptoalg",""{name=a}] & \calg\lar[shift left=5pt,"\algtohyp",""{name=b}] \ar[phantom, from=a, to=b,"\scriptstyle\bot"]
\etk\]
\end{thm}

\begin{proof}
  Let $\H$ be a hypergraph category; applying $\algtohyp \hyptoalg$ to $\H$ yields $$\H\xmapsto{\hyptoalg} (\FinSet\comma{\Ob\H}, \, \H(I,\Frob(-)))\xmapsto{\algtohyp} \H(I,\Frob(-))\Corel.$$

Observe that objects in $\H(I,\Frob(-))\Corel$ are lists in $\Ob\H$, and morphisms $X\to Y$ are the elements in $\H(X,Y)$; this allows us to define the unit of the adjunction as the natural transformation whose components are the functors
\begin{align*}
\eta_\H \colon \H & \to\H(I,\Frob(-))\Corel\\
X & \mapsto (X)
\end{align*}
mapping an object $X$ to the list with only one entry valued in $X$, and given by the identity on maps.

Now, let $(\C,F)$ be an object in $\calg$. Applying $\hyptoalg\algtohyp$ to $(\C, F)$ gives $$(\C,F)\xmapsto{\algtohyp} \FCorel \xmapsto{\hyptoalg} (\FinSet\comma{\Ob\FCorel}, \, \FCorel(I,\Frob(-))).$$
Since $\Ob\FCorel=\Ob\C$, and morphisms $I\to X$ in $\FCorel$ are the elements in $FX$, we let the counit $(\mon,\id)$ be given by the commutative diagram:
\[
  \btk
  \Cospan(\FinSet\comma{\Ob\FCorel}) \ar[r,"{\FCorel(I,\Frob(-))}"] \ar[d, "\Frob"']
  &[10ex] \Set \\
  \C \ar[ur, "F"']
  \etk
\]
where we note that the extension of the functor $\mon$ defined in Lemma~\ref{mon} to the domain $\Cospan(\FinSet\comma{\Ob\FCorel})$ is equal to $\Frob$.
Note that $\eta$ and $\epsilon$ are well-defined, since $\eta_\H$ is clearly a hypergraph functor for every $\H$, and $\mon$ is a finitely cocontinuous functor. It is easy to verify the naturality of $\eta$ and $\epsilon$, and that the triangle conditions are satisfied.
\end{proof}

\begin{rmk}
In the proof of Thm.\ \ref{A.and.phi} we can see that, for every hypergraph category $\H$, the corresponding component of the unit $$\eta_\H \colon \H\to\H(I,\Frob(-))\Corel$$ is an equivalence of hypergraph categories.

This means every hypergraph category is equivalent to $\FCorel$ for some choice of lax monoidal functor $F$. This fact ensures that $\calg$, and by extension, $\Data$, are sufficiently general to handle all hypergraph categories; an important fact for applications.
\end{rmk}

\subsection{Recovering the hypergraph category up to isomorphism}

Previously, we saw the unit of the adjunction $\hyptoalg \dashv \algtohyp$ only recovered a hypergraph category $\H$ up to equivalence. In the special case that $\H$ is \define{objectwise free}---meaning, is strict and has monoid of objects freely generated by some set $\Lambda$---\cite{FS} shows that we can use similar techniques to recover $\H$ up to \emph{isomorphism}. In this section we describe the relationship between these two constructions.

Write $\HypOF$ for the full subcategory of $\Hyp$ with objects those hypergraph categories that are objectwise free, and write $\CospanAlg$ for the full subcategory of $\calg$ with objects those cospan algebras with domain the category $\FinSet_\Lambda$ of $\Lambda$-labelled finite sets for some set $\Lambda$.

\begin{prop} \label{prop.ofequivalence}
  There exists an equivalence of categories
  \[
    \btk
      \HypOF \ar[r,shift left =3pt, "\hyptoalg'"] & \CospanAlg \ar[l,shift left=3pt, "\algtohyp'"]
    \etk
  \]
  such that
  \begin{equation} \label{eq.algtohyprestricts}
    \btk
    \HypOF \ar[r,hookrightarrow] & \Hyp\\
    \CospanAlg \ar[r,hookrightarrow] \ar[u,"\algtohyp'", "\cong"'] & \calg \ar[u,"\algtohyp"]
    \etk
  \end{equation}
  commutes and there exists a natural transformation
\[
  \btk
  \HypOF \ar[r,hookrightarrow] \ar[d,"A'"', "\cong"] & \Hyp \ar[d,"\hyptoalg"]\\
  \CospanAlg \ar[r,hookrightarrow] \ar[ur,Rightarrow,shorten >=25pt,shorten <=25pt, "\alpha"] & \calg
  \etk
\]
\end{prop}
\begin{proof}
  The functors $\algtohyp'$ and $\hyptoalg'$ witnessing the above equivalence of categories are given in Theorem~4.15 in \cite{FS}, where in that case they have the notation $\H_{-}$ and $A_{-}$ respectively.

  The commutativity of the square \eqref{eq.algtohyprestricts} states that the functor $\algtohyp'$ is a restriction of $\algtohyp\colon \calg \to \Hyp$ to the subcategory $\CospanAlg$. Referring to \cite{FS}, it is easy to observe that $\algtohyp'$ is indeed the restriction of $\algtohyp$ to the subcategory $\CospanAlg$ of $\calg$; here, we may take this as the definition of $\algtohyp'$. It is then straightforward to see that the image of $\algtohyp'$ lies in $\HypOF$, since the objects and monoidal structure of $\FCorel$ are inherited from the domain of $F$, and for every object of $\CospanAlg$ the domain is some objectwise free category $\FinSet_\Lambda$.

  The natural transformation $\alpha$ is defined as follows. Let $\H$ be a strict hypergraph category with objects generated by $\Lambda$. Then $\Lambda$ is some subset of $\Ob\H$, and this inclusion $g$ induces an inclusion functor $G\colon \FinSet\comma\Lambda \to \FinSet\comma{\Ob\H}$. On objects, the functor $\hyptoalg'$ is defined to map a hypergraph category $\H$ to the cospan algebra $(\FinSet_\Lambda, \H(I,-) \circ \frob_\H \circ G)$, and the component of the natural transformation $\alpha$ at $\H$ is given by the finite colimit preserving functor $G$ with the identity monoidal natural transformation
  \[
    \btk[row sep=5pt, column sep=10ex]
    \Cospan(\FinSet\comma\Lambda) \ar[rd,"{\H(I,-)\circ \frob_\H\circ G}" near start] \ar[dd, "G"'] \\
    & \Set \\
    \Cospan(\FinSet\comma{\Ob\H}) \ar[ur,"{\H(I,-)\circ\frob_\H}"' near start]
    \etk
  \]
\end{proof}

In particular, this means that if $\H$ has monoid of objects generated by some set $\Lambda$, then we apply $\hyptoalg'$ to obtain a representation of the hypergraph category as decorating data, and then taking decorated corelations on that data returns a hypergraph category isomorphic to $\H$.

\begin{cor} \label{cor.ofequivalence}
  Let $\H$ be an objectwise-free strict hypergraph category. Then there is an identity-on-objects isomorphism $\H \cong \hyptoalg'(\H)\Corel$.
\end{cor}

\section{An application: reaction networks}\label{section.application}

To conclude, we give a novel construction of the black box functor for reaction networks.

\subsection{The black box functor for reaction networks}

As mentioned in the introduction, the language of category theory and, in particular, of decorating data, is especially useful to express the semantics of {\it open} dynamical systems. As the name suggests, these consist of  dynamical systems which allow for a notion of inflow and outflow. In their work \cite{BP17}, Baez and Pollard assemble open dynamical systems into a decorated cospan category defined as follows.
\begin{defn}
  Let $\Dynam \coloneqq \DCospan$, where $D\colon \FinSet\to\Set$ is the lax monoidal functor which maps a finite set $X$ to the set $$DX \coloneqq \{v\colon \mathbb{R}^X\to\mathbb{R}^X \mid v \mbox{ is an algebraic vector field}\}$$ of all algebraic vector fields on $\mathbb{R}^X$.

A function $f\colon X\to Y$ is mapped by $D$ to the function
\begin{align*}
  Df\colon DX&\longrightarrow DY\\
  v&\longmapsto f_* \circ v \circ f^*
\end{align*}
where $f^* \colon \R^Y\to\R^X$ is the pullback, given by $f^* (c)(x)=c(f(x))$, and $f_* \colon \R^X\to\R^Y$ is the pushforward, defined as $f_*(c)(y)=\sum_{x\in f^{-1}(y)} c(x)$.
\end{defn}

An \define{open dynamical system} is a morphism $(X\xrightarrow{i} N \xleftarrow{o} Y, \, \R^N\xrightarrow{v}\R^N)$ in the category $\Dynam$, where $i\colon X\to N$ and $o\colon Y\to N$ mark the inflow and outflow variables, and $v$ is an algebraic vector field. Given an inflow $I\in\R^X$ and an outflow $O\in\R^Y$, the total flow of the system is given by the equation $v(c)+i_* I(c)-o_* O (c)$.

Just like with their non-open counterparts, it is of interest to study the {\it steady states} of an open system; the main difference being that now the possible action of inflows and outflows must also be taken into account. Given an open dynamical system as above, a steady state with inflows $I$ and outflows $O$ is an element $c\in\R^N$ satisfying
\begin{equation}
\label{eq.steady_state}
v(c)+i_* I(c)-o_* O (c)=0.
\end{equation}
The set of solutions to such equations form what is known as a semialgebraic relation.

A \define{semialgebraic subset} of the vector space $\R^n$ is a finite union of subsets of $\R^n$ of the forms $\{v \mid P(v) =0\}$ and $\{v \mid P(v)> 0\}$, where $P\colon \R^n \to \R$ is any polynomial $P(x_1,\dots,x_n)$ in the coordinates of $\R^n$. The property of being semialgebraic is invariant under linear transformations, and hence we can define semialgebraic subsets for any finite dimensional vector space $V$: they are simply subsets that are semialgebraic under any isomorphism $V \cong \R^n$.

A \define{semialgebraic relation} $\R^m \to \R^n$ is then a binary relation $S \subseteq \R^m \oplus \R^n$ that forms a semialgebraic subset. Semialgebraic relations are closed under composition of relations. We thus have a category.

\begin{defn}
  The category $\SemiAlgRel$ is the category with finite dimensional vector spaces as objects and semialgebraic relations between them as morphisms.
\end{defn}
Baez and Pollard study reaction networks by creating a black-box functor which takes an open dynamical system to the set of possible input and output flows and concentrations that yield open steady states of the system.

\begin{thm} \label{thm.baezpollard}
  There exists a symmetric monoidal functor $$\blacksquare\colon \Dynam \to \SemiAlgRel$$ taking a finite set $X$ to the vector space $\mathbb{R}^X\oplus\mathbb{R}^X$, and a morphism $(X\xrightarrow{i} N\xleftarrow{o}Y,\, v)$ to the semialgebraic subset of steady states
\[
  \{(i^* c,I,o^*c,O)\mid v(c)+i_* I-o_* O=0\} \subseteq \R^X\oplus\R^X\oplus\R^Y\oplus\R^Y.
\]
\end{thm}
Proving this theorem requires some work and, more importantly, some ingenuity.

In this final section, we show that Theorem~\ref{thm.baezpollard} is a corollary to a stronger, and more structured result (Cor.\ \ref{black_box}) that is easily derived from working within the framework of decorated corelations.  This stronger result constructs an objectwise free \emph{hypergraph} category $\sarel$ of semialgebraic relations and a hypergraph functor $\Dynam \to \sarel$. Instead of constructing the hypergraph functor directly, however, we find it is simpler to use the functor
  \[
    \HypOF\xrightarrow{\hyptoalg'}\calg\xhookrightarrow{\iota}\Data
  \]
  to find decorating data $S$ for $\sarel$, and then construct a morphism of decorating data from $D$ to $S$. Using the decorated corelations construction, we find that converting this morphism back to a hypergraph functor then easily implies the existence of the black box functor $\blacksquare$.

In short, recall the functors
\[
\btk[column sep=small]
\Hyp \ar[rr, shift left =5pt, "\hyptoalg"] &
  \scriptstyle\bot &
  \calg \ar[ll, shift left =5pt, "\algtohyp"] \ar[rr, shift left =5pt, "\iota"] &
  \scriptstyle\top &
  \Data.  \ar[ll,shift left =5pt, "\Kan"]
  \etk
\]
We find it easier to work on the right, in $\Data$, than on the left, in $\Hyp$.

\subsection{The hypergraph category of semialgebraic relations}

As mentioned above, we are interested in $\SemiAlgRel$ as the codomain for the black box functor. Note, however, that the image of $\blacksquare$ lies within the full subcategory of $\SemiAlgRel$ whose objects are of the form $\R^X\oplus\R^X$, for any finite set $X$.

\begin{defn}
  Let $\sarel$ be the full symmetric monoidal subcategory of $\SemiAlgRel$ whose objects are of the form $\R^X \oplus \R^X$ for some finite set $X$.
\end{defn}

We will see that this category is a more suitable semantic category for studying reaction networks.
The reason for this is that having restricted the objects to those in the image, $\sarel$ can now be equipped with a semantically meaningful hypergraph structure. Namely, recall that from the point of view of the reaction network, the two summands of the vector space $\R^X \oplus \R^X$ represent the spaces of concentrations and flows respectively. Concentrations and flows require different Frobenius structures to describe how they transform under coupling of reaction networks: we set concentrations equal, and add their flows. This is reminiscent of the case of potential and current in electric circuits and other passive linear networks in \cite{BF}, where the same structure is discussed in Sections~5.5 and 7.1.

More formally, we can obtain this hypergraph structure from the two different such structures present in $\LinRel$, as described in Example~\ref{linrel}. Observe that every linear subspace is semialgebraic, so there is an inclusion functor $\LinRel \hookrightarrow \SemiAlgRel$. Denote the two Frobenius structures on $\R\in\LinRel$ in Example~\ref{linrel} by $(\R,\mu_1,\eta_1,\delta_1,\epsilon_1)$ and $(\R,\mu_2,\eta_2,\delta_2,\epsilon_2)$. The object $\R\oplus\R$ can be given the Frobenius structure
\begin{equation} \label{eq.sarelfrobs}
  \big(\R \oplus \R, \:\: (\mu_1\oplus\mu_2)\circ \sigma_{2,3}, \:\: \eta_1\oplus\eta_2,\:\: \sigma_{2,3}\circ (\delta_1 \oplus\delta_2),\:\: \epsilon_1\oplus\epsilon_2\big)
\end{equation}
where $\sigma_{i,j}$ is the map that permutes the coordinates $i$ and $j$. In short, the object $\R\oplus\R$ is given the first Frobenius structure in the first coordinate, and the second in the second.

Then, since every object $\R^X \oplus \R^X$ in the category $\sarel$ is canonically isomorphic to the monoidal product of $\lvert X \rvert$ copies of $\R\oplus\R$, every object in $\sarel$ inherits a Frobenius structure from $\R \oplus \R$, and this equips $\sarel$ with the structure of a hypergraph category. Moreover, note that the objects of $\sarel$ can simply be considered finite sets, and hence the objects of $\FinSet$.

This yields the following proposition.
\begin{prop}
  With the hypergraph structure of \eqref{eq.sarelfrobs}, $\sarel$ is a strict hypergraph category with objects free on the one element set $1$.
\end{prop}

\subsection{A new viewpoint on the black box functor}
As $\sarel$ is objectwise free, we may apply the functor $\hyptoalg'$ of Proposition~\ref{prop.ofequivalence} to yield the cospan algebra
\[
  S\coloneqq \sarel(0,\frob(-))\colon \Cospan(\FinSet) \to \Set,
\]
where $0$ is the zero dimensional vector space. Note that $\frob$ sends a finite set $X$ to the object $\R^X \oplus \R^X$ of $\sarel$, and hence on objects this functor sends a finite set $X$ to the set of semialgebraic subsets of $\R^X \oplus \R^X$. On morphisms, loosely speaking, it sends a cospan of finite sets to its interpretation as Frobenius maps using the hypergraph structure of $\sarel$. Note that by Corollary~\ref{cor.ofequivalence}, we have an identity on objects isomorphism $S\Corel \cong \sarel$.

\begin{rmk}
  Note that while $S\Corel \cong \sarel$, there is a slight, but only cosmetic, difference in conventions for presenting morphisms. Indeed, in $S\Corel$ a morphism $X \to Y$ is specified by a subspace of $\R^{X+Y} \oplus \R^{X+Y}$, rather than a subspace of $\R^{X} \oplus \R^X \oplus \R^Y \oplus \R^Y$. 

  More subtly, due to the nature of the additive Frobenius structure on the second `flow' coordinate of each object in $\sarel$, composition in this second coordinate is given by summing to zero, rather than equality. That is, given decorations $U  \subseteq \R^{X+Y} \oplus \R^{X+Y}$ and $V  \subseteq \R^{Y+Z} \oplus \R^{Y+Z}$, their composite, as defined in Defn.~\ref{FCorel}, can be shown to be given by the subset
  \[
    \{(c_X,c_Z,f_X,f_Z) \mid \exists c_Y,f_Y. (c_X,c_Y,f_X,-f_Y) \in U, (c_Y,c_Z,f_Y,f_Z)\in V\}
  \]
  of $\R^{X+Z}\oplus \R^{X+Z}$. This difference with relational composition leads to the sign difference for $o_*O$ in equations \eqref{eq.steady_state} above and \eqref{eq.scorel} below: in the former $o_\ast O$ describes net flow \emph{out}, while in the latter $o_\ast O$ describes net flow \emph{in} (and hence has a negative sign).
\end{rmk}

We now have two objects of $\Data$: the data $(\FinSet, \, (\FinSet,\I), \, D)$ from which $\Dynam$ is constructed, and the data $(\FinSet, \, (\I,\FinSet), \, S)$ from which $\sarel$ is constructed. By abuse of notation we'll just refer to these objects as $D$ and $S$ respectively.

We now wish to define a morphism of decorating data $D \to S$. As we know, this is given by a pair $(A,\alpha)$ consisting of a finitely cocontinuous functor $A\colon \FinSet\to\FinSet$ and a natural transformation $\alpha\colon D \Rightarrow S \circ A$.

\begin{prop}
  For all finite sets $X$, define $\alpha_X\colon DX \to SX$ to be the function
  \[
    \alpha_X\colon \big\{v\colon \R^X\to\R^X \,\big|\, v \textrm{ algebraic} \big\}
    \longrightarrow
    \big\{R \subseteq \R^X\oplus\R^X \,\big|\, R \textrm{ semialgebraic}\big\}
  \]
sending an algebraic vector field $v$ to its graph $\graph(v) =\{(c,v(c))\}\subseteq\R^X\oplus\R^X$. This defines a monoidal natural transformation
\[\btk
\FinSet\dar[hookrightarrow, "\ \ \ \Downarrow\alpha" near start]\rar["D"] & \Set \\
\Cospan(\FinSet)\ar[ru,"S"'] &
\etk\]
\end{prop}
\begin{proof}
Since $v$ is algebraic, its graph $\graph(v)\coloneqq \{(c,v(c))\}\subseteq\R^X\oplus\R^X$ forms a semialgebraic subset, and thus $\alpha$ is well defined. Showing that $\alpha$ is natural in $\FinSet$ amounts to proving that, for any function $f\colon X\to Y$ and any algebraic vector field $v \in DX$ we have
\[
\frob(f)\graph(v)=Sf \circ \alpha_X(v)= \alpha_Y\circ Df (v)= \graph(f_* v f^*)
\]
in the set $SY$ of semialgebraic subsets of $\R^Y \oplus \R^Y$, where we abuse notation to also write $f$ for the cospan $X \xrightarrow{f} Y \xleftarrow{\id_Y} Y$.

To prove this, we must more closely examine the Frobenius structure in $\sarel$. In particular, we must understand the semialgebraic (and in fact linear) relation $\frob(f)\colon \R^X \oplus \R^X \to \R^Y \oplus \R^Y$. It is not difficult to show that this is given by the notion of pushforward and pullback, so that
\[
  \frob(f) = \{(f^\ast b,a,b,f_\ast a)\} \subseteq \R^X \oplus \R^X \oplus \R^Y \oplus \R^Y.
\]
See for example \cite[\textsection 5.5, \textsection 7.1]{BF}. This immediately implies naturality:
\[
  \frob(f)\graph(v)=\{(d,f_\ast v(c)) \mid f^\ast d=c\} = \{(d,f_\ast v f^\ast(d))\}= \graph(f_* v f^*). \qedhere
\]
\end{proof}

Applying the functor $(-)\Corel$ to the morphism $(\id,\alpha)\colon D \to S$ gives the following corollary.

\begin{cor}\label{black_box}
  There exists a hypergraph functor
  \[
    \blacksquare\colon \Dynam \to S\Corel \cong \sarel.
  \]
  On objects, this functor maps a finite set $X$ to the vector space $\R^X \oplus \R^X$. On morphisms, it maps a $D$-decorated cospan $(X\xrightarrow{i} N \xleftarrow{o} Y, \, v)$ to the decoration that results from the composition
  \begin{equation} \label{eq.scorel}
    v\xmapsto{\alpha} \{(c, v(c))\}\xmapsto{S[i,o]^\op} \{(i^* c, o^* c, I,O ) \mid v(c)=i_* I + o_* O \}.
\end{equation}
\end{cor}
\begin{proof}
  Recall that the action of the functor $(-)\Corel$ on morphisms is given in Prop.~\ref{prop.deccorelfunctor}. In particular, note that 
  \[
    S[i,o] = \frob[i,o] = \{(i^*c,o^*c,I,O,c,i_*I+o_*O)\} \subseteq \R^{X+Y}\oplus \R^{X+Y} \oplus \R^N \oplus \R^N,
  \]
  yielding the map \eqref{eq.scorel}.
\end{proof}

This functor is, up to a restriction of the codomain to $\sarel \subseteq \SemiAlgRel$, equal to black box functor $\blacksquare$ constructed in \cite{BP17}.

\newpage

\appendix
\section{The explicit computation of $\Kan$} \label{appendix}

\begin{proof}[Proof of Proposition \ref{prop.kanexplicit}]
Verifying that $\kappa$ is a natural transformation is virtually immediate; instead, we focus on proving that the universal property of the left Kan extension is satisfied.

For this, suppose there exists a functor $G \colon \Cospan(\C)\to\Set$ and a natural transformation $\gamma$ as on the right below; we wish to show there exists a unique $\beta \colon \Lan F\Rightarrow G$ such that we have the equation
\[\btk[row sep=large]
\C\cp \M^\op \dar[hookrightarrow, "\ \ \ \Downarrow\kappa" near start ]\rar["F"] & \Set \\
\Cospan(\C)\ar[ru,"\Lan F" description]\ar[ru,"G"',"\Downarrow\beta" pos=1/3, bend right=40pt]
\etk \ \ \ = \ \ \
\btk[row sep=large]
\C\cp \M^\op \rar["F"]\dar[hookrightarrow,"\ \ \ \Downarrow\gamma" near start] & \Set \\
\Cospan(\C)\ar[ru,"G"']
\etk\]
Given an object $X$ in $\C$, let $\beta_X \colon \Lan FX\to GX$ be the function $$(X\xrightarrow{e} N, s\in FN)\mapsto G(e^\op )\gamma_N(s).$$
Then, for every $X$, the composition $\beta_X\kappa_X$ is the correspondence $$s\xmapsto{\kappa_X} \btk[column sep=small] (X\rar[equal] & X, s)\etk \xmapsto{\beta_X} G(\id^\op )\gamma_X(s)=\gamma_X(s)$$ and so the above equality is satisfied.

To show that $\beta$ is natural, let $f=(X\xrightarrow{i} M \xleftarrow{o} Y)$ be a map in $\Cospan(\C)$, and $(X\xrightarrow{e} N, s)$ an element in $\Lan FX$. Denoting the composition by $fe^\op =(N\xrightarrow{j_N} N+_X M\xleftarrow{j_M o} Y)$ and the factorisation of the right leg of this cospan by $Y\xrightarrow{e'} \overline{N+_X M}\xrightarrow{m} N+_X M$, we see that
\begin{align*}
(X\xrightarrow{e} N,\, s)
&\xmapsto{\Lan F (f)} \big(Y\xrightarrow{e'} \overline{N+_X M},\, F(m^\op )F(j_N)(s)\big) \\
&\xmapsto{\mathmakebox[6.5ex]{\beta_Y}} G({e'}^\op )\gamma_{\overline{N+_X M}}F(m^\op )F(j_N)(s).
\end{align*}
On the other hand, we have
$$(X\xrightarrow{e} N, s)\xmapsto{\beta_X} G(e^\op )\gamma_N(s)\xmapsto{G(f)} G(f)G(e^\op )\gamma_N(s).$$
The naturality of $\gamma$, together with the fact that the cospans $fe^\op $ and ${e'}^\op m^\op j_N$ belong to the same isomorphism class, imply these two compositions yield the same result.

To prove uniqueness, assume there exists another natural transformation $\beta' \colon \Lan F\Rightarrow G$ such that $\btk[column sep=small] \beta'_X(X\rar[equal] & X, s)\etk =\gamma_X(s)$ for every $X$. Consider any other element $(X\xrightarrow{e} N,t)$ in $\Lan F X$; using the naturality of $\beta'$, and chasing the element $\btk[column sep=small] (N\rar[equal] & N, t)\etk$ along the commutative diagram
\[\btk
\Lan FN\rar["\beta'_N"]\dar["\Lan F (e^\op )"'] & GN\dar[" G(e^\op )"]\\
\Lan FX\rar["\beta'_X"'] & GX
\etk\] gives $$\beta'_X(X\xrightarrow{e} N,t)=G(e^\op )\gamma_N(t)$$ which is, by definition, equal to  $\beta_X(X\xrightarrow{e} N, t)$. This implies $\beta'_X=\beta_X$ for every $X$, which concludes our proof.
\end{proof}
\begin{proof}[Proof of Proposition \ref{Kan.on.maps}]
According to Proposition \ref{def.kan}, $\beta$ is the (unique) natural transformation such that
\[\btk[column sep=large]
\C\cp\M^\op \ar[dr,"F"]\dar[hookrightarrow,"\iota"', "\ \ \ \ \Downarrow\kappa" near end] &  \\
\Cospan(\C)\dar[hookrightarrow,"A"', "\ \ \ \ \Downarrow\beta" near start]\rar["\Lan F" description] & \Set\\
\Cospan(\C')\ar[ru,"\Lan F'"'] &
\etk \ \ = \ \
\btk
\C\cp\M^\op \ar[dr,"F"]\dar[hookrightarrow,"\iota"'] &  \\
\Cospan(\C)\rar[phantom, "\scriptstyle{\Downarrow\kappa'\alpha}" near start]\dar[hookrightarrow,"A"'] & \Set\\
\Cospan(\C')\ar[ru,"\Lan F'"'] &
\etk\]
Following the proof of Proposition \ref{prop.kanexplicit}, we have $G=\Lan F' A$ and $\gamma=\kappa'\alpha \colon F\Rightarrow \Lan F' A$, which is defined as
\[\btk[column sep=small]
\gamma_X(s)=\kappa'_{AX}\alpha_X(s)=(AX\rar[equal] & AX, \alpha_X(s)).
\etk\]
Thus, $\beta \colon \Lan F\Rightarrow\Lan F' A$ is given on its $X$-component by
\[\btk[column sep=small]
(X\xrightarrow{e} N, s\in FN)\mapsto \Lan F' A(e^\op )(AN\rar[equal] & AN, \alpha_N(s)).
\etk\]

To compute the right-hand side, we factor the right leg of the cospan $\id_{AN}Ae^\op =Ae^\op$ in $(\E',\M')$ as shown in the triangular diagram above, obtaining $Ae=m'e'$; then, the function $\Lan F' (Ae^\op )$ takes the element $\btk[column sep=small]
(AN\rar[equal] & AN, \alpha_N(s))
\etk$ to $$(AX\xrightarrow{e'} \overline{AN}, F'({m'}^\op )\alpha_N(s))$$ as claimed.
\end{proof}


\begin{thebibliography}{WWW00}
\bibliographystyle{alpha}
\bibitem[ASW11]{ASW11} L.\ de F.\ Albasini, N.\ Sabadini, R.\ F.\ C.\ Walters, The compositional construction of Markov processes, {\sl Applied Categorical Structures}, {\bf 19}(1):425--437 (2011).
  \href{https://doi.org/10.1007/s10485-010-9233-0}{doi:10.1007/s10485-010-9233-0}.

  \bibitem[BF18]{BF} J.\ C.\ Baez and B.\ Fong, A compositional framework for
  passive linear networks.  Available as
  \href{https://arxiv.org/abs/1504.05625}{arXiv:1504.05625}.


  \bibitem[BFP16]{BFP16} J.\ C.\ Baez, B.\ Fong and B.\ S.\ Pollard, A
  compositional framework for Markov processes, {\sl J.\ Math.\ Phys.}, {\bf 57}
  (2016), 033301.  Also available as
  \href{https://arxiv.org/abs/1508.06448}{arXiv:1508.06448}.

  \bibitem[BP17]{BP17} J.\ C.\ Baez and B.\ S.\ Pollard, A compositional
  framework for reaction networks, {\sl Rev.\ Math.\ Phys.} {\bf 29} (2017),
  1750028.  Also available as
  \href{https://arxiv.org/abs/1704.02051}{arXiv:1704.02051}.

    \bibitem[Ben67]{Ben67} J.\ B\'enabou, Introduction to bicategories I. In
    {\sl Reports of the Midwest Category Seminar}, Lecture Notes in Mathematics
    47:1--77, Springer, 1967.


\bibitem[BSZ17]{BSZ17} F.\ Bonchi, P.\ Sobocinski, F.\ Zanasi, The Calculus of Signal Flow Diagrams I: Linear Relations on Streams, {\sl Information and Computation}, {\bf 252}:2--29 (2017).
  \href{https://doi.org/10.1016/j.ic.2016.03.002}{doi:10.1016/j.ic.2016.03.002}.

    \bibitem[Car91]{Car91} A.\ Carboni, Matrices, relations and group
      representations, {\sl J.\ Algebra}, {\bf 138}:497--529 (1991).

   \bibitem[Fon15]{Fong15} B.\ Fong, Decorated cospans, {\sl Theory Appl.\ Categ.}, {\bf
   30} (2015), 1096--1120. Available at
   \href{http://www.tac.mta.ca/tac/volumes/30/33/30-33abs.html}{http://www.tac.mta.ca/tac/volumes/30/33/30-33abs.html}.

  \bibitem[Fon16]{Fong16} B.\ Fong, \textsl{The Algebra of Open and Interconnected Systems},
     Ph.D.\ thesis, Department of Computer Science, University of Oxford, 2016.
     Available as \href{http://arxiv.org/abs/1609.05382}
    {arXiv:1609.05382}.

   \bibitem[Fon18]{Fong18} B.\ Fong, Decorated corelations, {\sl Theory Appl.\ Categ.},
     {\bf 33} (2018), 608--643. Available at
     \href{http://www.tac.mta.ca/tac/volumes/33/22/33-22abs.html}{http://www.tac.mta.ca/tac/volumes/33/22/33-22abs.html}.

   \bibitem[FS18]{FS} B.\ Fong and D.\ I.\ Spivak, Hypergraph categories.
   Available as \href{https://arxiv.org/abs/1806.08304}{arXiv:1806.08304}.

  \bibitem[FP18]{FP} T.\ Fritz and P.\ Perrone, A criterion for Kan
  extensions of lax monoidal functors. Available as
  \href{https://arxiv.org/abs/1809.10481}{arXiv:1809.10481}.

 \bibitem[GKS17]{GKS17} A.\ Gianola, S.\ Kasangian, N.\ Sabadini,
   Cospan/Span(Graph): an algebra for open, reconfigurable automata networks, {\sl CALCO 2017}, 2:1-2:17 (2017). \href{https://doi.org/10.4230/LIPIcs.CALCO.2017.2}{doi:10.4230/LIPIcs.CALCO.2017.2}.

 \bibitem[Joh77]{Joh77} P.\ T.\ Johnstone, {\sl Topos Theory}, Academic Press, New York, 1977.

	\bibitem[KSW00]{KSW} P.\ Katis, N.\ Sabadini, R.\ F.\ C.\ Walters, On the algebra
      of systems with feedback and boundary, {\sl Rendiconti del Circolo
      Matematico di Palermo Serie II, Suppl.} {\bf 63} (2000), 123--156.

    \bibitem[Mac98]{Mac} S. Mac Lane, {\sl Categories for the Working Mathematician}, 2nd ed., Springer, New York, 1998.

    \bibitem[RSW05]{RSW05} R.\ Rosebrugh, N.\ Sabadini and R.\ F.\ C.\ Walters, Generic
      commutative separable algebras and cospans of graphs, \textsl{Th.\ App.\
      Cat.\ }{\bf 15} (2005), 264--277. Available at
      \href{http://www.tac.mta.ca/tac/volumes/15/6/15-06abs.html}{http://www.tac.mta.ca/tac/volumes/15/6/15-06abs.html}.

\bibitem[RSW08]{RSW08} R.\ Rosebrugh, N.\ Sabadini and R.\ F.\ C.\ Walters,
Calculating colimits compositionally, in P.\ Degano {\sl et al.}, {\sl Concurrency, Graphs
and Models}, Lecture Notes in Computer Science, vol.\ 5065, Springer,
Berlin, 2008, pp.\ 581--592.  Also available as \href{https://arxiv.org/abs/0712.2525}
{arXiv:0712.2525}.

\end{thebibliography}
\end{document}